\documentclass[10pt,a4paper]{article}
\usepackage[utf8]{inputenc}
\usepackage{amsmath}
\usepackage{amsfonts}
\usepackage{amssymb}
\author{Andreas Basse-O'Connor, Mikkel Slot Nielsen,
Jan Pedersen, \\ and Victor Rohde \\
Department of Mathematics \\
Aarhus University \\
\{basse, mikkel, jan, victor\}@math.au.dk
}
\date{}

\usepackage{amssymb}
\usepackage{amsthm}
\usepackage{amscd}
\usepackage{amstext}
\usepackage{fancyhdr}

\usepackage{amsthm}
\usepackage[numbers]{natbib}
\usepackage[T1]{fontenc}
\usepackage{enumerate}
\usepackage[american]{babel}
\usepackage{graphicx}
\usepackage{bbm}
\usepackage[mathscr]{euscript}
\usepackage{mathrsfs}
\usepackage{stmaryrd}
\usepackage{soul}
\usepackage{hyperref}
\usepackage{xspace}
\usepackage[draft,danish]{fixme}
\usepackage{mathtools}
\usepackage{dsfont}
\usepackage{url}
\usepackage[ps,all,dvips]{xy}
\SelectTips{cm}{12}
\CompileMatrices

  \makeatletter
 \useshorthands{"}
 \defineshorthand{"-}{\nobreak-\bbl@allowhyphens}
 \makeatother

\newcommand{\FF}{\mathcal{F}} 
 
\newcommand{\EE}{\mathbb{E}} 
\DeclareMathOperator{\Real}{Re}

\newcommand{\RR}{\ensuremath{\mathbb{R}}}           
\newcommand{\CC}{\ensuremath{\mathbb{C}}}

\newcommand{\LL}{\mathcal{L}}

\theoremstyle{plain}
\newtheorem{lemma}{Lemma}[section]
\newtheorem{proposition}[lemma]{Proposition}
\newtheorem{theorem}[lemma]{Theorem}
\newtheorem{corollary}[lemma]{Corollary}
\theoremstyle{definition}
\newtheorem{example}[lemma]{Example}

\theoremstyle{definition}

\theoremstyle{remark}

\newtheorem{remark}[lemma]{Remark}

\newcommand{\devnull}[1]{}

\numberwithin{equation}{section}

\begin{document}


\title{Multivariate stochastic delay differential equations and CAR representations of \\
CARMA processes}






\maketitle

\begin{abstract}

In this study we show how to represent a continuous time autoregressive moving average (CARMA) as 
a higher order stochastic delay differential equation, which may be thought of as a continuous-time equivalent of the AR($\infty$) representation. Furthermore, we show how this representation gives rise to a prediction formula for CARMA processes. To be used in the above mentioned results we develop a general theory for multivariate stochastic delay differential equations, which will be of independent interest, and which will have particular focus on existence, uniqueness and representations. 
%
%
\\
\\
\noindent \footnotesize \textit{AMS 2010 subject classifications:} 60G05; 60G22; 60G51; 60H05; 60H10
\\
\\
\noindent \textit{Keywords:} multivariate stochastic delay differential equations; multivariate Ornstein-Uhlenbeck processes; CARMA processes; FICARMA processes; MCARMA processes; noise recovery; prediction; long memory

\end{abstract}





\section{Introduction and  main ideas}

The class of autoregressive moving averages (ARMA) is one of the most popular classes of stochastic processes for modeling time series in discrete time. This class goes back to the thesis of Whittle in 1951 and was popularized in \citet{Box-Jenkins}. The continuous time analogue of an ARMA process is called a CARMA process, and it is the formal solution $(X_t)_{t\in \RR}$ to the equation 
\begin{align}\label{MCARMAheureqn}
P(D) X_t = Q(D) DZ_t,\quad t \in \mathbb{R},
\end{align}
where $P$ and $Q$ are polynomials of degree $p$ and $q$, respectively. Furthermore, $D$ denotes differentiation with respect to $t$, and $(Z_t)_{t\in \RR}$ is a L\'evy process, the continuous time analogue of a random walk. In the following we will assume that $p>q$ and $P(z),Q(z)\neq 0$ whenever $\text{Re}(z)\geq 0$. In this case one can give precise meaning to $(X_t)_{t\in \RR}$ as a causal stochastic process through a state-space representation as long as $(Z_t)_{t\in \mathbb{R}}$ has log moments. L\'{e}vy-driven CARMA processes have found many applications, for example, in modeling temperature, electricity and stochastic volatility, cf. \cite{BenthBenthPrice,GarciaKluppelbergMullerElectricity,TodorovStochasticVol}. Moreover, there exists a vast amount of literature on theoretical results for CARMA processes (and variations of these), and a few references are \cite{brockwellRecent,BrockwellMarquardt,BrockwellLevyCARMA,nonNegCARMA,marquardtMFICARMA,MarquardtStelzer,stelzer2011carma}.

It is well-known that any causal CARMA process has a continuous time moving average representation of CMA($\infty$) type
\begin{align*}
X_t = \int_{-\infty}^t g(t-u)\,dZ_u,\qquad t\in\RR, 
\end{align*}
see the references above or Section~\ref{CARMArelation}. This representation may be very convenient for studying many of their properties. A main contribution of our work is that we obtain a CAR($\infty$) representation of CARMA processes of the form
\begin{align}\label{ARrep}
R(D)X_t = \int_0^\infty X_{t-u}f(u)\, du  + DZ_t,\quad t \in \mathbb{R},
\end{align} 
where $R$ is a polynomial of order $p-q$ and $f:\RR\to \RR$ is a deterministic function, both defined through $P$ and $Q$. Since $(X_t)_{t\in \mathbb{R}}$ is $p-q-1$ times differentiable, see \cite[Proposition~3.32]{MarquardtStelzer}, the relation \eqref{ARrep} is well-defined if we integrate both sides once. A heuristic argument for obtaining \eqref{ARrep} from \eqref{MCARMAheureqn} is as follows. If $q=0$, $Q$ is constant and \eqref{ARrep} holds with $R=P$ and $f =0$. If $q\geq 1$, it is convenient to rephrase \eqref{MCARMAheureqn} in the frequency domain:
\begin{align}\label{frequencyEq}
\frac{P(-iy)}{Q(-iy)}\mathcal{F}[X](y) = \mathcal{F}[DL](y),\quad y  \in \mathbb{R}.
\end{align}
Using polynomial long division we may choose a polynomial $R$ of order $p-q$ such that
\begin{align*}
S(z):= Q(z) R(z)-P(z),\qquad z\in \CC, 
\end{align*}
is a polynomial of at most order $q-1$. Now observe that
\begin{align*}
\frac{P(-iy)}{Q(-iy)}\mathcal{F}[X](y)
&= \biggr(R(-iy) - \frac{S(-iy)}{Q(-iy)} \biggr) \mathcal{F}[X](y)\\
&= \mathcal{F}[R(D)X](y) - \mathcal{F}[f](y) \mathcal{F}[X](y),
\end{align*}
where $f:\mathbb{R}\to \mathbb{R}$ is the $L^2$ function characterized by $\mathcal{F}[f](y) = S(-iy)/Q(-iy)$ for $y\in \mathbb{R}$. (In fact, we even know that $f$ is vanishing on $(-\infty,0)$ and decays exponentially fast at $\infty$, cf. Remark~\ref{fComputation}.) Combining this identity with \eqref{frequencyEq} results in the representation \eqref{ARrep}.

We show in Theorem~\ref{MCARMAasMSDDE} that \eqref{ARrep} does indeed hold true for any invertible (L\'{e}vy-driven) CARMA process. Similar relations are shown to hold for invertible fractionally integrated CARMA (FICARMA) processes, where $(Z_t)_{t\in \mathbb{R}}$ is a fractional L\'{e}vy process, and also for their multi-dimensional counterparts, which we will refer to as MCARMA and MFICARMA processes, respectively. We use these representations to obtain a prediction formula for general CARMA type processes (see Corollary~\ref{CARMAprediction}). A prediction formula for invertible one-dimensional L\'{e}vy-driven CARMA processes is given in \cite[Theorem 2.7]{brockwell2015prediction}, but prediction of MCARMA processes has, to the best of our knowledge, not been studied in the literature.

Autoregressive representations such as \eqref{ARrep} are useful for several reasons. To give a few examples, they separate the noise $(Z_t)_{t\in \RR}$ from $(X_t)_{t\in\RR}$ and hence provide a recipe for recovering increments of the noise from the observed process, they ease the task of prediction (and thus estimation), and they clarify the dynamic behavior of the process. These facts motivate the idea of defining a broad class of processes, including the CARMA type processes above, which all admit an autoregressive representation, and it turns out that a well-suited class to study is the one formed by solutions to multi-dimensional stochastic delay differential equations (MSDDEs). To be precise, for an integrable $n$-dimensional (measurable) process $Z_t=(Z_t^1,\dots,Z_t^n)^T$, $t \in \RR$, with stationary increments and a finite signed $n \times n$ matrix-valued measure $\eta$, concentrated on $[0,\infty)$, a stationary process $X_t = (X^1_t,\dots, X^n_t)^T$, $t \in \mathbb{R}$, is a solution to the associated MSDDE if it satisfies
\begin{align}\label{MultiSDDEcompact}
dX_t = \eta \ast X (t)\, dt + dZ_t.
\end{align}
By equation \eqref{MultiSDDEcompact} we mean that
\begin{align}\label{MSDDE}
X^j_t -X^j_s= \sum_{k=1}^n \int_s^t \int_{[0,\infty)} X^k_{u-v}\, \eta_{jk}(dv)\, du + Z^j_t - Z^j_s , \quad j = 1,\dots, n,
\end{align}
almost surely for each $s<t$. This system of equations is an extension of the stochastic delay differential equation (SDDE) in \cite[Section~3.3]{contARMAframework} to the multivariate case. The overall structure of \eqref{MultiSDDEcompact} is also in line with earlier literature such as \cite{GK,Mohammed} on univariate SDDEs, but here we allow for infinite delay ($\eta$ is allowed to have unbounded support) which is a key property in order to include the CARMA type processes in the framework.



The structure of the paper is as follows: In Section~\ref{prel} we introduce the notation used throughout this paper. Next, in Section~\ref{sdfsdf}, we develop the general theory for MSDDEs with particular focus on existence, uniqueness and prediction.  The general results of Section~\ref{sdfsdf} are then specialized in Section~\ref{sectionMSDDE} to various settings. Specifically, in Section~\ref{RegularIntegrator} we consider the case where the noise process gives rise to a reasonable integral, and in Section~\ref{hOrderSection} we demonstrate how to derive results for higher order SDDEs by nesting them into MSDDEs. Finally, in Section~\ref{CARMArelation} we use the above mentioned findings to represent CARMA processes and generalizations thereof as solutions to higher order SDDEs and to obtain the corresponding prediction formulas.

\section{Notation}\label{prel}
Let $f:\mathbb{R}\to \mathbb{C}^{m\times k}$ be a measurable function and $\mu$ a $k \times n$ (non-negative) matrix measure, that is, 
\begin{align*}
\mu = \begin{bmatrix} \mu_{11} & \cdots & \mu_{1n} \\ \vdots & \ddots & \vdots \\
\mu_{k1} & \cdots & \mu_{kn}
\end{bmatrix}
\end{align*}
where each $\mu_{jl}$ is a measure on $\mathbb{R}$. Then, we will write $f\in L^p(\mu)$ if
\begin{align*}
\int_\mathbb{R}\vert f_{il}(u) \vert^p  \mu_{lj}  (du) < \infty
\end{align*}
for $l=1,\dots, k$, $i=1,\dots, m$ and $j=1,\dots, n$. Provided that $f\in L^1(\mu)$, we set
\begin{align}\label{fIntegral}
\int_\mathbb{R} f(u)\, \mu (du) = \sum_{l=1}^k\begin{bmatrix}  \int_\mathbb{R} f_{1l}(u)\, \mu_{l1}(du) & \cdots &  \int_\mathbb{R} f_{1l}(u)\, \mu_{ln}(du) \\
\vdots & \ddots & \vdots \\
 \int_\mathbb{R} f_{ml}(u)\, \mu_{l1}(du) & \cdots &  \int_\mathbb{R} f_{ml}(u)\, \mu_{ln}(du)
\end{bmatrix}.
\end{align}
If $\mu$ is the Lebesgue measure, we will suppress the dependence on the measure and write $f \in L^p$, and in case $f$ is measurable and bounded Lebesgue almost everywhere, $f \in L^\infty$. For two (matrix) measures $\mu^+$ and $\mu^-$ on $\mathbb{R}$, where at least one of them are finite, we call the set function $\mu (B) := \mu^+ (B) - \mu^- (B)$, defined for any Borel set $B$, a signed measure (and, from this point, simply referred to as a measure). We may and do assume that the two measures $\mu^+$ and $\mu^-$ are singular. To the measure $\mu$ we will associate its variation measure $\vert \mu\vert := \mu^+ + \mu^-$, and when $\vert \mu \vert (\mathbb{R})< \infty$, we will say that $\mu$ is finite. Integrals with respect to $\mu$ are defined in a natural way from \eqref{fIntegral} whenever $f\in L^1 (\mu):=L^1(\vert \mu \vert)$. If $f$ is one-dimensional, respectively if $\mu$ is one-dimensional, we will write $f\in L^1(\mu)$ if $f \in L^1 (\vert \mu_{ij} \vert)$ for all $i=1,\dots, k$ and $j=1,\dots, n$, respectively if $f_{ij}\in L^1(\vert \mu \vert)$ for all $i=1,\dots,m$ and $j=1,\dots, k$. The associated integral is defined in an obvious manner. 

We define the convolution at a given point $t\in \mathbb{R}$ by  
\begin{align*}
f \ast \mu (t) = \int_\mathbb{R} f(t-u)\mu (du)
\end{align*}
provided that $f(t-\cdot)\in L^1( \mu )$. In case that $\mu$ is the Lebesgue-Stieltjes measure of a function $g:\mathbb{R}\to \mathbb{R}^{k\times n}$ we will also write $f\ast g (t)$ instead of $f\ast \mu (t)$ (not to be confused with the standard convolution between functions). For a given measure $\mu$ we set 
\begin{align*}
D(\mu) = \biggr\{ z \in \mathbb{C}\, :\, \int_\mathbb{R}e^{\text{Re}(z)u}\, \vert \mu_{ij}\vert (du)< \infty \quad \text{ for } i=1,\dots, k\  \text{and}\ j=1,\dots, n  \biggr\}
\end{align*}
and define its Laplace transform $\mathcal{L}[\mu]$ as
\begin{align*}
\mathcal{L}[\mu]_{ij}(z) = \int_\mathbb{R}e^{zu}\, \mu_{ij} (du), \quad \text{for} \quad i=1,\dots,k, \: j=1,\dots,n,
\end{align*}
for every $z \in D(\mu)$. If $\mu$ is a finite measure, we will also refer to the Fourier transform $\mathcal{F}[\mu]$ of $\mu$, which is given as $\mathcal{F}[\mu](y) = \mathcal{L}[\mu](iy)$ for $y \in \mathbb{R}$. If $\mu (du) = f(u)\, du$ for some measurable function $f$, we write $\mathcal{L}[f]$ and $\mathcal{F}[f]$ instead. We will also use that the Fourier transform $\mathcal{F}$ extends from $L^1$ to $L^1 \cup L^2$, and it maps $L^2$ onto $L^2$. We will say that $\mu$ has a moment of order $p \in \mathbb{N}_0$ if
\begin{align*}
\int_\mathbb{R}\vert u \vert^p\, \vert\mu_{jk}\vert (du)< \infty
\end{align*}
for all $j,k = 1,\dots, n$. Finally, for two functions $f,g:\mathbb{R}\to \mathbb{R}$ and $a \in [-\infty,\infty]$, we write $f (t) = o(g(t))$, $f(t) \sim g(t)$ and $f(t) = O(g(t))$ as $t \to a$ if
\begin{align*}
\lim_{t\to a} \frac{f(t)}{g(t)}\to 0,\quad \lim_{t\to a}\frac{f(t)}{g(t)} = 1\quad \text{and}\quad \limsup_{t\to a} \biggr\vert\frac{f(t)}{g(t)} \biggr\vert < \infty,
\end{align*} 
respectively.

\section{Stochastic delay differential equations}\label{sdfsdf}



Consider the general MSDDE in \eqref{MultiSDDEcompact}, where the noise $(Z_t)_{t\in \mathbb{R}}$ is a measurable process, which is integrable and has stationary increments. The first main result provides sufficient conditions to ensure existence and uniqueness of a solution. To obtain such results we need to put assumptions on the delay measure $\eta$. In order to do so, we associate to $\eta$ the function $h :  D(\eta) \to \CC^{n\times n}$ given by 
\begin{align}\label{DefOfh}
h(z) = -zI_n - \LL[\eta](z).
\end{align} 
where $I_n$ is the $n\times n$ identity matrix.

\begin{theorem}\label{existence} Let $h$ be given in \eqref{DefOfh} and suppose that $\det (h(iy)) \neq 0$ for all $y \in \RR$. Suppose further that $\eta$ has second moment. Then there exists a function $g:\mathbb{R}\to \mathbb{R}^{n \times n}$ in $L^2$ characterized by
\begin{align}\label{gKernelChar1}
\mathcal{F}[g](y) = h(iy)^{-1},
\end{align}
the convolution
\begin{align}\label{solutionForm}
g \ast Z (t) := Z_t + \int_\mathbb{R} g \ast \eta (t-u)\, Z_u\, du
\end{align}
is well-defined for each $t \in \mathbb{R}$ almost surely, and $X_t = g \ast Z (t)$, $t \in \mathbb{R}$, is the unique (up to modification) stationary and integrable solution to \eqref{MultiSDDEcompact}. If, in addition to the above stated assumptions, $\det (h(z)) \neq 0$ for all $z \in \CC$ with $\Real (z) \leq 0$ then the solution in \eqref{solutionForm} is casual in the sense that $(X_t)_{t\in \mathbb{R}}$ is adapted to the filtration
\begin{align*}
 \{\sigma (Z_t - Z_s\, :\, s<t)\}_{t \in \mathbb{R}}.
\end{align*}
\end{theorem}
The solution $(X_t)_{t\in \mathbb{R}}$ to \eqref{MultiSDDEcompact} will very often take form as a $(Z_t)_{t\in \mathbb{R}}$-driven moving average, that is,
\begin{align}\label{noiseMA}
X_t = \int_\mathbb{R}g(t-u)\, dZ_u
\end{align}
for each $t \in \mathbb{R}$ (cf. Section~\ref{RegularIntegrator}). This fact justifies the notation $g\ast Z$ introduced in \eqref{solutionForm}. In case $n=1$, equation \eqref{MultiSDDEcompact} reduces to the usual first order SDDE, and then the existence condition becomes $h(iy) = -iy - \mathcal{F}[\eta](y) \neq 0$ for all $y\in \mathbb{R}$, and the kernel driving the solution is characterized by $\mathcal{F}[g](y) = 1/h(iy)$. This is consistent with earlier literature (cf. \cite{contARMAframework,GK,Mohammed}). 

The second main result concerns prediction of MSDDEs. In particular, the content of the result is that we can compute a prediction of future values of the observed process if we are able to compute the same type of prediction of the noise.

\begin{theorem}\label{Prediction} Suppose that $\det (h(z)) \neq 0$ for all $z \in \CC$ with $\text{Re}(z)\leq 0$ and that $\eta$ has second moment. Furthermore, let $(X_t)_{t\in \mathbb{R}}$ be the stationary and integrable solution to \eqref{MultiSDDEcompact} and let $g$ be given by \eqref{gKernelChar1}. Fix $s<t$. Then, if we set
\begin{align}\label{PredictionOfNoise}
\hat{Z}_u = \mathbb{E}[Z_u-Z_s \mid Z_s -Z_r,\, r< s],\quad u >s,
\end{align}
it holds that
\begin{align*}
\MoveEqLeft \EE [X_t \mid X_u, \, u\leq s] \\
&=  g(t-s) X_s + \int_s^t g(t-u) \eta \ast \big\{\mathds{1}_{(-\infty,s]}X\big\} (u) \, du + g \ast \big\{\mathds{1}_{(s,\infty)}\hat{Z} \big\}(t),
\end{align*} 
using the notation
\begin{align*}
\big(\eta \ast \{\mathds{1}_{(-\infty,s]}X\}(u)\big)_j &:= \sum_{k=1}^n \int_{[u-s,\infty)} X^k_{u-v}\, \eta_{jk}(dv) \quad \text{and} \\
\big(g \ast \{\mathds{1}_{(s,\infty)}\hat{Z}\}(u)\big)_j &:= \sum_{k=1}^n \int_{[0,u-s)} \hat{Z}^k_{u-v}\, g_{jk}(dv)
\end{align*}
for $u>s$ and $j=1,\dots,n$.
\end{theorem}

\begin{remark}
In case $(Z_t)_{t\in \mathbb{R}}$ is a L\'{e}vy process, the prediction formula in Theorem~\ref{Prediction} simplifies, since $\hat{Z}_u = (u-s)\mathbb{E}[Z_1]$ and thus
\begin{align*}
\MoveEqLeft \EE [X_t \mid X_u, \, u\leq s] \\
&=  g(t-s) X_s + \int_s^t g(t-u) \eta \ast \big\{\mathds{1}_{(-\infty,s]}X\big\} (u) \, du + \int_s^tg(t-u)\, du\, \mathbb{E}[Z_1],
\end{align*}
using integration by parts. Obviously, the formula takes an even simpler form if $\mathbb{E}[Z_1]= 0$. If instead we are in a long memory setting and $(Z_t)_{t\in \mathbb{R}}$ is a fractional Brownian motion, we can rely on \cite{GripenbergNorros} to obtain $(\hat{Z}_u)_{s<u\leq t}$ and then use the formula given in Theorem~\ref{Prediction} to compute the prediction $\mathbb{E}[X_t\mid X_u,\, u\leq s]$. 
\end{remark}
In Section~\ref{CARMArelation} we use this prediction formula combined with the relation between MSDDEs and MCARMA processes to obtain a prediction formula for any invertible MCARMA process.

%
%

\section{Examples and further results}\label{sectionMSDDE}


In this section we will consider several examples of MSDDEs and give some additional results. We begin by defining what we mean by a regular integrator, since this makes it possible to have the compact form \eqref{noiseMA} of the solution to \eqref{MultiSDDEcompact} in most cases. Next, we show how one can nest higher order MSDDEs in the (first order) MSDDE framework. Finally, we show that invertible MCARMA processes (and some generalizations) form a particular subclass of solutions to higher order MSDDEs.  

\subsection{Regular integrators and moving average representations}\label{RegularIntegrator}

When considering the form of the solution in Theorem~\ref{existence} it is natural to ask if this can be seen as a moving average of the kernel $g$ with respect to the noise $(Z_t)_{t \in \RR}$, that is, if
\begin{align}\label{movingAverageEntry}
X^j_t = \bigg(\int_\mathbb{R}g (t-u)\, dZ_u\biggr)_j =  \sum_{k=1}^n \int_\mathbb{R}g_{jk}(t-u)\, dZ^k_u , \quad t \in \mathbb{R},
\end{align}
for $j=1,\dots, n$. The next result shows that the answer is positive if $(Z^k_t)_{t\in \mathbb{R}}$ is a "reasonable"\ integrator for a suitable class of deterministic integrands for each $k=1,\dots, n$.

\begin{proposition}\label{MArep}
Let $h$ be the function given in \eqref{DefOfh} and suppose that, for all $y \in \RR$, $\det (h(iy)) \neq 0$. Suppose further that $\eta$ has second moment and let $(X_t)_{t\in \mathbb{R}}$ be the solution to \eqref{MultiSDDEcompact} given by \eqref{solutionForm}. Finally assume that, for each $k=1, \dots, n$, there exists a linear map $I_k:L^1 \cap L^2 \to L^1(\mathbb{P})$ which has the following properties:
\begin{enumerate}[(i)]
\item For all $s<t$, $I_k (\mathds{1}_{(s,t]}) = Z^k_t - Z^k_s$.

\item If $\mu$ is a finite Borel measure on $\mathbb{R}$ having first moment then
\begin{align}\label{FubiniRelation}
I_k \biggr( \int_\mathbb{R}f_r(t-\cdot)\, \mu (dr) \biggr) = \int_\mathbb{R}I_k(f_r(t-\cdot))\, \mu (dr) 
\end{align}
almost surely for all $t\in \mathbb{R}$, where $f_r= \mathds{1}_{[0,\infty)}(\cdot-r) - \mathds{1}_{[0,\infty)}$ for $r \in \mathbb{R}$.
\end{enumerate}
Then it holds that
\begin{align}\label{MArelation}
X^j_t = \sum_{k=1}^n I_k (g_{jk}(t-\cdot)), \quad j=1,\dots, n,
\end{align}
almost surely for each $t\in \mathbb{R}$. In this case, $(Z_t)_{t \in \RR}$ will be called a regular integrator and we will write $\int \cdot \, dZ^k = I_k$.
\end{proposition}
The typical example of a regular integrator is a multi-dimensional L\'{e}vy process:
\begin{example}\label{LevyIntegrator}
Suppose that $(Z_t)_{ t\in \RR}$ is an $n$-dimensional integrable L\'{e}vy process. Then, in particular, each $(Z^j_t)_{t\in \mathbb{R}}$ is an integrable (one-dimensional) L\'{e}vy process, and in \cite[Lemma~5.3]{contARMAframework} it is shown that the integral $\int_\mathbb{R} f(u)\, dZ^j_u$ is well-defined in the sense of \cite{Rosinski_spec} and belongs to $L^1(\mathbb{P})$ if $f \in L^1 \cap L^2$. Moreover, the stochastic Fubini result given in \cite[Theorem~3.1]{QOU} implies in particular that condition (ii) of Proposition~\ref{MArep} is satisfied, which shows that $(Z_t)_{t\in \mathbb{R}}$ is a regular integrator and that \eqref{movingAverageEntry} holds. 
\end{example}
We will now show that a class of multi-dimensional fractional L\'{e}vy processes can serve as regular integrators as well (cf. Example~\ref{fracLevyIntegrator} below). Fractional noise processes are often used as a tool to incorporate (some variant of) long memory in the corresponding solution process. As will appear, the integration theory for fractional L\'{e}vy processes we will use below relies on the ideas of \cite{Tina}, but is extended to allow for symmetric stable L\'{e}vy processes as well. For more on fractional stable L\'{e}vy processes, the so-called linear fractional stable motions, we refer to \cite[p. 343]{Stable}. First, however, we will need the following observation:
\begin{proposition}\label{RiemannReg}
Let $f:\mathbb{R}\to \mathbb{R}$ be a function in $L^1\cap L^{\alpha}$ for some $\alpha \in (1,2]$. Then the right-sided Riemann-Liouville fractional integral
\begin{align}\label{RiemannLiouvilleFrac}
I^{\beta}_- f: t \mapsto \frac{1}{\Gamma (\beta)}\int_t^\infty f(u)(u-t)^{\beta -1}\, du
\end{align}
is well-defined and belongs to $L^\alpha$ for any $\beta \in (0,1-1/\alpha)$.
\end{proposition}

\begin{example}\label{fracLevyIntegrator}
Let $\alpha = (\alpha_1,\dots, \alpha_n)$ with $\alpha_j \in (1,2]$ and $f=(f_{jk}):\mathbb{R}\to \mathbb{R}^{n\times n}$ be a function such that $f_{jk} \in L^1 \cap L^{\alpha_k}$ for $j,k=1,\dots, n$. Consider an $n$-dimensional L\'{e}vy process $(L_t)_{t\in \mathbb{R}}$ where its $j$-th coordinate is symmetric $\alpha_j$-stable if $\alpha_j \in (1,2)$ and mean zero and square integrable if $\alpha_j=2$. Then, for a given vector $\beta = (\beta_1,\dots, \beta_n)$ with $\beta_j \in (0,1-1/\alpha_j)$ for $j=1,\dots, n$ the corresponding fractional L\'{e}vy process $(Z_t)_{t\in \mathbb{R}}$ with parameter $\beta$ is defined as
\begin{align*}
Z^j_t &= \int_\mathbb{R}\big(I^{\beta_j}_- [\mathds{1}_{(-\infty,t]}-\mathds{1}_{(-\infty,0]}]\big)(u)\, dL^j_u \\
&= \frac{1}{\Gamma (1+ \beta_j)}\int_\mathbb{R}\big[(t-u)_+^{\beta_j} - (-u)^{\beta_j}_+\big]\, dL^j_u
\end{align*}
for $t\in \mathbb{R}$ and $j=1,\dots, n$, and where $x_+ = \max \{x,0\}$. In light of Proposition~\ref{RiemannReg}, this definition makes it natural to define the integral of a function $f:\mathbb{R}\to \mathbb{R}$ in $L^1 \cap L^{\alpha_j}$ (particularly in $L^1 \cap L^2$) with respect to $(Z^j_t)_{t\in \mathbb{R}}$ as
\begin{align*}
\int_\mathbb{R} f(u)\, dZ^j_u = \int_\mathbb{R} \big(I^{\beta_j}_-f\big)(u)\, dL^j_u
\end{align*}
for $j=1,\dots, n$. Note that the integral belongs to $L^2(\mathbb{P})$ for $\alpha_j=2$ and to $L^\gamma(\mathbb{P})$ for any $\gamma <\alpha_j$ if $\alpha_j \in (1,2)$. Using Proposition~\ref{RiemannReg} and the stochastic Fubini result given in \cite[Theorem~3.1]{QOU} for $(L^j_t)_{t\in \mathbb{R}}$ it is straightforward to verify that assumption (ii) of Proposition~\ref{MArep} is satisfied as well, and thus $(Z_t)_{t \in \RR}$ is a regular integrator and the solution $(X_t)_{t\in \mathbb{R}}$ to \eqref{MultiSDDEcompact} takes the moving average form \eqref{movingAverageEntry}.
\end{example}

At this point it should be clear that the conditions for being a regular integrator are mild, hence they will, besides the examples mentioned above, also be satisfied for a wide class of semimartingales with stationary increments.

\subsection{Higher order (multivariate) SDDEs}\label{hOrderSection}

An advantage of introducing the multivariate setting \eqref{MultiSDDEcompact} is that we can nest higher order MSDDEs in this framework. Effectively, as usual and as will be demonstrated below, it is done by increasing the dimension accordingly. 

Let $\varpi_0,\varpi_1,\dots, \varpi_{m-1}$ be (entrywise) finite $n \times n$ measures concentrated on $[0,\infty)$ which all admit second moment, and let $(Z_t)_{t\in \mathbb{R}}$ be an $n$-dimensional integrable stochastic process with stationary increments. For convenience we will assume that $(Z_t)_{t\in \mathbb{R}}$ is a regular integrator in the sense of Proposition~\ref{MArep}. We will say that an $n$-dimensional stationary, integrable and measurable process $(X_t)_{t\in \mathbb{R}}$ satisfies the corresponding $m$-th order MSDDE if it is $m-1$ times differentiable and
\begin{align}\label{hOrderSDDE}
dX^{(m-1)}_t = \sum_{j=0}^{m-1} \varpi_j\ast X^{(j)}  (t)\, dt + dZ_t
\end{align}
where $(X^{(j)}_t)_{t\in \mathbb{R}}$ denotes the entrywise $j$-th derivative of $(X_t)_{t\in\mathbb{R}}$ with respect to $t$. By \eqref{hOrderSDDE} we mean that
\begin{align*}
\big(X^{(m-1)}_t\big)^k-\big(X^{(m-1)}_s\big)^k = \sum_{j=0}^{m-1}\sum_{l=1}^n \int_s^t \int_{[0,\infty)} \big(X^{(j)}_{u-v}\big)^l\, (\varpi_j )_{kl} (dv)\, du +Z^k_t - Z^k_s
\end{align*}
for $k=1,\dots, n$ and each $s<t$ almost surely.
Equation \eqref{hOrderSDDE} corresponds to the $mn$-dimensional MSDDE in \eqref{MultiSDDEcompact} with noise $(0,\dots, 0,Z_t^T)^T \in \mathbb{R}^{m n}$ and
\begin{align}\label{etaHigherOrder}
\eta = \begin{bmatrix*}[c] 0 & I_n\delta_0 & 0 & \cdots & 0 \\
0 & 0 & I_n\delta_0 & \cdots & 0 \\
\vdots & \vdots & \vdots & \ddots & \vdots \\
0 & 0 & 0 &  \cdots & I_n\delta_0 \\
\varpi_0 & \varpi_1 & \varpi_2 & \cdots & \varpi_{m-1}
\end{bmatrix*}.
\end{align}
(If $n=1$ then $\eta = \varpi_0$.) 
With $\eta$ given by \eqref{etaHigherOrder} it follows that
\begin{align*}
D(\eta) = \bigcap_{j=0}^{m-1} D(\varpi_j)
\end{align*}
and
\begin{align*}
h(z) = -\begin{bmatrix}
I_n z & I_n & 0 & \cdots & 0 \\
0 & I_n z & I_n & \cdots & 0 \\
\vdots & \vdots & \ddots & \ddots & \vdots \\
0 & 0 & \cdots & I_n z &  I_n\\
\mathcal{L}[\varpi_0](z) & \mathcal{L}[\varpi_1](z) & \cdots & \mathcal{L}[\varpi_{m-2}](z) & I_n z+\mathcal{L}[\varpi_{m-1}](z)
\end{bmatrix}
\end{align*}
for $z \in D(\eta)$. In general, we know from Theorem~\ref{existence} that a solution to \eqref{hOrderSDDE} exists if $\det (h(iy))\neq 0$ for all $y \in \mathbb{R}$, and in this case the unique solution is given by
\begin{align}\label{solutionTohOrder}
X_t = \int_\mathbb{R}g_{1m}(t-u)\, dZ_u,\quad t \in \mathbb{R},
\end{align}
where $\mathcal{F}[g_{1m}]$ is characterized as entrance $(1,m)$ in the $n \times n$ block representation of $h(i\cdot)^{-1}$. In other words, if $e_j$ denotes the $j$-th canonical basisvector of $\mathbb{R}^m$ and $\otimes$ the Kronecker product,
\begin{align*}
\mathcal{F}[g_{1m}](y) = (e_1\otimes I_n)^T h(iy)^{-1}(e_m\otimes I_n)
\end{align*}
for $y \in \mathbb{R}$. However, due to the particular structure of $\eta$ in \eqref{etaHigherOrder} we can simplify these expressions:
\begin{theorem}\label{simpleFourierKernel} Let the setup be as above. Then it holds that
\begin{align}\label{higherOrderDetCond}
\det (h(z)) = \det \biggr(I_n(-z)^m - \sum_{j=0}^{m-1} \mathcal{L}[\varpi_j](z) (-z)^j \biggr)
\end{align}
for all $z\in D(\eta)$, and if $\det (h(iy))\neq 0$ for all $y \in \mathbb{R}$, there exists a unique solution to \eqref{hOrderSDDE} and it is given as \eqref{solutionTohOrder} where $g:\mathbb{R}\to \mathbb{R}^{n \times n}$ is characterized by
\begin{align}\label{hOrderFourierT}
\mathcal{F}[g_{1m}](y) = \biggr( I_n(-iy)^m - \sum_{j=0}^{m-1} \mathcal{F}[\varpi_j](y) (-iy)^j \biggr)^{-1}
\end{align}
for $y \in \mathbb{R}$. The solution is causal if $\det (h(z))\neq 0$ whenever $\text{Re}(z)\leq 0$.
\end{theorem}
Observe that, as should be the case, we are back to the first order MSDDE when $m=1$ and \eqref{higherOrderDetCond}-\eqref{hOrderFourierT} agree with Theorem~\ref{existence}. As we will see in Section~\ref{CARMArelation} below, one motivation for introducing higher order MSDDEs of the form \eqref{hOrderSDDE} and to study the structure of the associated solutions, is their relation to MCARMA processes. However, we start with the multivariate CAR($p$) process, where no delay term will be present, as an example:
\begin{example}\label{CARprocess}
Let $P(z) = I_nz^p + A_1z^{p-1}+ \cdots + A_p$, $z \in\mathbb{C}$, for suitable $A_1,\dots, A_p\in \mathbb{R}^{n \times n}$. The associated CAR($p$) process $(X_t)_{t\in \mathbb{R}}$ with noise $(Z_t)_{t\in \mathbb{R}}$ can be thought of as formally satisfying $P(D)X_t = DZ_t$, $t\in \mathbb{R}$, where $D$ denotes differentiation with respect to $t$. Integrating both sides and rearranging terms gives
\begin{align}\label{CARrelation}
dX^{(p-1)}_t = - \sum_{j=0}^{p-1}A_{p-j}X^{(j)}_t\, dt + dZ_t,\quad t \in \mathbb{R},
\end{align}
which is of the form \eqref{hOrderSDDE} with $m=p$ and $\varpi_j = -A_{p-j}\delta_0$ for $j=0,1,\dots, p-1$. Proposition~\ref{simpleFourierKernel} shows that a unique solution exists if
\begin{align*}
\det\biggr( I_n(iy)^p + \sum_{j=0}^{p-1} A_{p-j} (iy)^j \biggr) = \det (P(iy))\neq 0
\end{align*}
for all $y \in \mathbb{R}$, and in this case $\mathcal{F}[g_{1m}](y) = P(-iy)^{-1}$ for $y\in \mathbb{R}$. This agrees with the rigorous definition of the CAR($p$) process, see e.g. \cite{MarquardtStelzer}. In case $p=1$, \eqref{CARrelation} collapses to the multivariate Ornstein-Uhlenbeck equation
\begin{align*}
dX_t = - A_1 X_t\, dt + dZ_t, \quad t \in \mathbb{R},
\end{align*}
and if the eigenvalues of $A_1$ are all positive, it is easy to check that $g_{1m}(t) = e^{-A_1t}\mathds{1}_{[0,\infty)}(t)$ so that the unique solution $(X_t)_{t\in \mathbb{R}}$ is causal and takes the well-known form
\begin{align}\label{OUsolution}
X_t = \int_{-\infty}^te^{-A_1(t-u)}\, dZ_u
\end{align}
for $t\in \mathbb{R}$. Lévy-driven multivariate Ornstein-Uhlenbeck processes have been studied extensively in the literature, and the moving average structure \eqref{OUsolution} of the solution is well-known when $(Z_t)_{t \in \RR}$ is a Lévy process. We refer to \cite{BarndorffJensenSoerensen,SatoWatanabeYamazato,SatoYamazato} for further details. The one-dimensional case where $(Z_t)_{t \in \RR}$ is allowed to be a general stationary increment process has been studied in \cite{QOU}.
\end{example}

\subsection{Relations to MCARMA processes}\label{CARMArelation}
Let $p \in \mathbb{N}$ and define the polynomials $P,Q:\mathbb{C}\to \mathbb{C}^{n\times n}$ by
\begin{align}\label{CARMApolyn}
\begin{aligned}
P(z) &= I_nz^p + A_1 z^{p-1}+ \cdots + A_p \quad \text{and}\\
Q(z) &= B_0 + B_1z + \cdots + B_{p-1}z^{p-1}
\end{aligned}
\end{align} 
for $z \in \mathbb{C}$ and suitable $A_1,\dots, A_p,B_0,\dots,B_{p-1} \in \RR^{n \times n}$. We will also fix $q \in \mathbb{N}_0$, $q<p$, and set $B_q = I_n$ and $B_j = 0$ for all $q<j<p$. It will always be assumed that $\det (P(iy)) \neq 0$ for all $y \in \mathbb{R}$. Under this assumption there exists a function $\tilde{g}:\mathbb{R}\to \mathbb{R}^{n \times n}$ which is in $L^1\cap L^2$ and 
\begin{align}\label{definingMCARMAkernel}
\mathcal{F}[\tilde{g}](y)= P(-iy)^{-1}Q(-iy)
\end{align}
for every $y \in \mathbb{R}$. Consequently, for any regular integrator $(Z_t)_{t\in \mathbb{R}}$ in the sense of Proposition~\ref{MArep}, the $n$-dimensional stationary and integrable process $(X_t)_{t \in \mathbb{R}}$ given by
\begin{align}\label{ZdrivenCARMA}
X_t = \int_\mathbb{R} \tilde{g}(t-u)\, dZ_u, \quad t \in \mathbb{R},
\end{align}
is well-defined. If it is additionally assumed that $\det (P(z)) \neq 0$ for $z \in \CC$ with $\Real (z) \geq 0$ then it is argued in \cite{MarquardtStelzer} that 
\begin{align}\label{gMCARMA}
\tilde{g}(t) = \mathds{1}_{[0,\infty)}(t)(e_1^p\otimes I_n)^T e^{At}E
\end{align} 
where
\begin{align*}
A = \begin{bmatrix}
0 & I_n & 0 & \cdots & 0 \\
0 & 0 & I_n &\cdots & 0\\
\vdots & \vdots & \ddots & \ddots & \vdots \\
0 & 0 & \cdots & 0 & I_n \\
-A_p & -A_{p-1} & \cdots & -A_{2} & -A_{1}
\end{bmatrix} \quad \text{and}\quad E= \begin{bmatrix}
E_1 \\ \vdots \\ E_p
\end{bmatrix},
\end{align*}
with $E(z) = E_1z^{p-1} + \cdots + E_p$ chosen such that
\begin{align*}
z \mapsto P(z) E(z) -Q(z)z^p
\end{align*}
is at most of degree $p-1$. (Above, and henceforth, we use the notation $e^k_j$ for the $j$-th canonical basis vector of $\RR^k$.) We will refer to the process $(X_t)_{t\in \mathbb{R}}$ as a $(Z_t)_{t\in \mathbb{R}}$-driven MCARMA($p,q$) process. For instance, when $(Z_t)_{t \in \RR}$ is an $n$-dimensional L\'{e}vy process, $(X_t)_{t\in \mathbb{R}}$ is a (L\'{e}vy-driven) MCARMA($p,q$) process as introduced in \cite{MarquardtStelzer}. If $(L_t)_{t\in \mathbb{R}}$ is an $n$-dimensional square integrable L\'{e}vy process with mean zero, and
\begin{align*}
Z^j_t = \frac{1}{\Gamma (1+\beta_j)}\int_\mathbb{R} \big[(t-u)_+^{\beta_j}-(-u)_+^{\beta_j}\big]\, dL^j_u, \quad t\in \mathbb{R},
\end{align*}
for $\beta_j \in (0,1/2)$ and $j=1,\dots, n$, then $(X_t)_{t\in \mathbb{R}}$ is an MFICARMA($p,\beta,q$) process, $\beta = (\beta_1,\dots, \beta_n)$, as studied in \cite{marquardtMFICARMA}. For the univariate case ($n=1$), the processes above correspond to the CARMA($p,q$) and FICARMA($p,\beta_1,q$) process, respectively. The class of CARMA processes has been studied extensively, and we refer to the references in the introduction for details.

\begin{remark} Observe that, generally, L\'{e}vy-driven MCARMA (hence CARMA) processes are defined even when $(Z_t)_{t\in \mathbb{R}}$ has no more than log moments. However, it relies heavily on the fact that $\tilde{g}$ and $(Z_t)_{t\in \mathbb{R}}$ are well-behaved enough to ensure that the process in \eqref{ZdrivenCARMA} remains well-defined. At this point, a setup where the noise does not admit a first moment has not been integrated in a framework as general as that of \eqref{MultiSDDEcompact}.
\end{remark}

In the following our aim is to show that, under a suitable invertibility assumption, the $(Z_t)_{t\in \mathbb{R}}$-driven MCARMA($p,q$) process given in \eqref{ZdrivenCARMA} is the unique solution to a certain (possibly higher order) MSDDE of the form \eqref{hOrderSDDE}. Before formulating the main result of this section we introduce some notation. To $P$ and $Q$ defined in \eqref{CARMApolyn} we will associate the unique polynomial $R(z) = I_n z^{p-q}+C_{p-q-1}z^{p-q-1} + \cdots +C_0$, $z \in \mathbb{C}$ and $C_0,C_1,\dots, C_{p-q-1}\in \mathbb{R}^{n\times n}$, having the property that
\begin{align}\label{residuePolyn}
z \mapsto Q(z)R(z) - P(z)
\end{align}
is a polynomial of at most order $q-1$ (see the introduction for an intuition about why this property is desirable). 

\begin{theorem}\label{MCARMAasMSDDE} Let $P$ and $Q$ be given as in \eqref{CARMApolyn}, and let $(X_t)_{t\in \mathbb{R}}$ be the associated $(Z_t)_{t\in \mathbb{R}}$-driven MCARMA($p,q$) process. Suppose that $\det (Q(z))\neq 0$ for all $z\in \CC$ with $\Real (z)\geq0$. Then $(X_t)_{t\in \mathbb{R}}$ is the unique solution to \eqref{hOrderSDDE} with 
\begin{align*}
m=p-q, \quad\varpi_0 (du) = -C_0\delta_0(du) + f(u)\, du , \quad \text{and} \quad \varpi_j= -C_j\delta_0, 
\end{align*}
for $1 \leq j \leq m-1$ or, written out,
\begin{align}\label{specificRelation}
dX^{(m-1)}_t = -\sum_{j=0}^{m-1} C_j X^{(j)}_t\, dt + \biggr(\int_0^\infty X_{t-u}^Tf(u)^T\, du\biggr)^T\, dt + dZ_t,
\end{align}
where $C_0,\dots, C_{m-1}\in \mathbb{R}^{n \times n}$ are defined as in \eqref{residuePolyn} above, $(X^{(j)}_t)_{t\in \mathbb{R}}$ is the $j$-th derivative of $(X_t)_{t\in \mathbb{R}}$, and where $f : \RR \to \RR^{n\times n}$ is characterized by 
\begin{align}\label{fFourier}
\FF[f](y) = R(-iy)-Q(-iy)^{-1}P(-iy).
\end{align}
\end{theorem}
It follows from Theorem~\ref{MCARMAasMSDDE} that $p-q$ is the order of the (possibly multivariate) SDDE we can associate with a (possibly multivariate) CARMA process. Thus, this seems as a natural extension of \cite{contARMAframework}, where the univariate first order SDDE is studied and related to the univariate CARMA($p,p-1$) process.

\begin{remark}\label{CARMAasSDDEassump}
An immediate consequence of Theorem~\ref{MCARMAasMSDDE} is that we obtain an inversion formula for $(Z_t)_{t\in\mathbb{R}}$-driven MCARMA processes. In other words, it shows how to recover the increments of $(Z_t)_{t\in \mathbb{R}}$ from observing $(X_t)_{t\in \mathbb{R}}$. For this reason it seems natural to impose the invertibility assumption $\det (Q(z))\neq 0$ for all $z \in\mathbb{C}$ with $\text{Re}(z)\geq 0$, which is the direct analogue of the one for discrete time ARMA processes (or, more generally, moving averages). It is usually referred to as the minimum phase property in signal processing. The inversion problem for (L\'{e}vy-driven) CARMA processes has been studied in  \cite{contARMAframework,brockwellRecent, nonNegCARMA,brockwell2015prediction} and for (L\'{e}vy-driven) MCARMA processes in \cite{BrockwellSchlemmMCARMAinversion}. In both cases a different approach, which does not rely on MSDDEs, is used.  
\end{remark}

\begin{remark}\label{fComputation}
Since the Fourier transform $\mathcal{F}[f]$ of the function $f$ defined in Theorem~\ref{MCARMAasMSDDE} is rational, one can determine $f$ explicitly (e.g., by using the partial fraction expansion of $\mathcal{F}[f]$). Indeed, since the Fourier transform of $f$ is of the same form as the Fourier transform of the solution kernel $\tilde{g}$ of the MCARMA process we can deduce that
\begin{align}\label{fMCARMA}
f(t) = (e_1^q \otimes I_n)^T e^{Bt} F, \quad t \geq 0,
\end{align}
with
\begin{align*}
B = \begin{bmatrix}
0 & I_n & 0 & \cdots & 0 \\
0 & 0 & I_n &\cdots & 0\\
\vdots & \vdots & \ddots & \ddots & \vdots \\
0 & 0 & \cdots & 0 & I_n \\
-B_0 & -B_1 & \cdots & -B_{q-2} & -B_{q-1}
\end{bmatrix} \quad \text{and}\quad F= \begin{bmatrix}
F_1 \\ \vdots \\ F_q
\end{bmatrix},
\end{align*}
where $F(z) = F_1z^{q-1} + \cdots + F_q$ is chosen such that
\begin{align*}
z \mapsto Q(z) F(z) - [Q(z)R(z) - P(z)]z^q
\end{align*}
is at most of degree $q-1$ (see \eqref{definingMCARMAkernel} and \eqref{gMCARMA}).
\end{remark}

In Corollary~\ref{CARMAprediction} we formulate the prediction formula in Theorem \ref{Prediction} in the special case where $(X_t)_{t \in \RR}$ is a $(Z_t)_{t \in \RR}$-driven MCARMA process. In the formulation we use the definition 
\begin{align*}
\hat{Z}_u  = \mathbb{E}[Z_u-Z_s \mid Z_s -Z_r,\, r< s],\quad u >s,
\end{align*}
in line with \eqref{PredictionOfNoise}.

\begin{corollary}\label{CARMAprediction}
Let $(X_t)_{t \in \RR}$ be a $(Z_t)_{t \in \RR}$-driven MCARMA process and set
\begin{align*}
\tilde{g}_j(t) = (e_1^p \otimes I_n)^T e^{At} \sum_{k=j}^{p-q} A^{k-j}EC_k, \quad t \geq 0,
\end{align*}
for $j=1,\dots, p-q$, where $C_0,\dots,C_{p-q-1}$ are given in \eqref{residuePolyn} and $C_{p-q}=I_n$. Suppose that $\det (P(z))\neq 0$ and $\det (Q(z))\neq 0$ for all $z\in \CC$ with $\Real (z)\geq0$. Fix $s<t$. Then the following prediction formula holds 
\begin{align*}
\MoveEqLeft \EE [X_t \mid X_u, \, u\leq s] =   \sum_{j=1}^{p-q} \tilde{g}_j(t-s) X_s^{(j-1)} \\
&+ \int_{-\infty}^s \int_s^t \tilde{g}(t-u)  f(u-v) \, du \, X_v \, dv +  \tilde{g} \ast \{\hat{Z}\mathds{1}_{(s,\infty)} \}(t),
\end{align*}
where $\tilde{g}$ and $f$ are given in \eqref{gMCARMA} and \eqref{fMCARMA}, respectively, and
\begin{align*}
\tilde{g} \ast \{\hat{Z}\mathds{1}_{(s,\infty)} \}(t) = \mathds{1}_{\{p=q+1\}} \hat{Z}_u + (e_1^p \otimes I_n)^TAe^{At}\int_s^t e^{-Av}E \hat{Z}_v \, dv.
\end{align*}

\end{corollary}



\begin{example}

To illustrate the results above we will consider an $n$-dimensional $(Z_t)_{t \in \RR}$-driven MCARMA($3$,$1$) process $(X_t)_{t \in \RR}$ with $P$ and $Q$ polynomials given by 
\begin{align*}
\begin{aligned}
P(z) &= I_nz^3 + A_1 z^{2}+ A_2 z + A_3,\\
Q(z) &= B_0 + I_nz 
\end{aligned}
\end{align*} 
for matrices $B_0,A_1, A_2,A_3\in \mathbb{R}^{n \times n}$ such that $\det (P(z)) \neq 0$ and $\det (Q(z))\neq 0$ for all $z \in \mathbb{C}$ with $\text{Re}(z)\geq 0$. According to \eqref{gMCARMA}, $(X_t)_{t\in \mathbb{R}}$ may be written as
\begin{align*}
X_t = \int_{-\infty}^t (e_1^3\otimes I_n)^Te^{A(t-u)} E \, dZ_u
\end{align*}
where $E_1 = 0$, $E_2=I_n$, and $E_3 = B_0 - A_1$. With 
\begin{align*}
C_1 =  A_1-B_0, \quad C_0 = A_2 + B_0( B_0 - A_1)
\end{align*}
and 
\begin{align*}
F = B_0(A_2 - B_0 ( A_1- B_0))-A_3,
\end{align*}
Theorem~\ref{MCARMAasMSDDE} and Remark \ref{fComputation} imply that 
\begin{align*}
 dX^{(1)}_t =& - C_1 X^{(1)}_t\, dt -C_0 X_t\, dt  + \biggr(\int_0^\infty (FX_{t-u})^T e^{-B_0^Tu}\, du\biggr)^T\, dt + dZ_t.
\end{align*}
Moreover, by Corollary~\ref{CARMAprediction}, we have the prediction formula 
\begin{align*}
 \EE [X_t \mid X_u, \, u\leq s]  =& (e_1^3\otimes I_n)^T e^{At}\biggr[ (EC_1 + AE) X_s + E X_s^{(1)} \\
&+\int_s^t e^{-Au} E \biggr(e^{B_0 u}\int_{-\infty}^s e^{-B_0 v} F X_v\, dv + \hat{Z}_u \biggr)\, du\biggr].
\end{align*}
\end{example}

\section{Proofs and auxiliary results}\label{proofs}

We will start this section by discussing some technical results. These results will then be used in the proofs of all the results stated above.

Recall the function $h:D(\eta) \to \mathbb{C}^{n \times n}$ defined in \eqref{DefOfh}. Note that we always have  $\{z \in \mathbb{C}\, :\, \text{Re}(z) \leq 0\}\subseteq D(\eta)$ and $h(iy) = -iyI_n-\mathcal{F}[\eta](y)$ for $y\in \mathbb{R}$. Provided that $\eta$ is sufficiently nice, Proposition~\ref{gExistence} below ensures the existence of a kernel $g:\mathbb{R}\to \mathbb{R}^{n \times n}$ which will drive the solution to \eqref{MultiSDDEcompact}.

\begin{proposition}\label{gExistence} Let $h$ be given as in \eqref{DefOfh} and suppose that $\det (h(iy))\neq 0$ for all $y \in \mathbb{R}$. Then there exists a function $g=(g_{jk}):\mathbb{R}\to \mathbb{R}^{n \times n}$ in $L^2$ characterized by
\begin{align}\label{gKernelChar}
\mathcal{F}[g](y) = h(iy)^{-1}
\end{align}
for $y \in \mathbb{R}$. Moreover, the following statements hold:
\begin{enumerate}[(i)]
\item\label{functionRelation} The function $g$ satisfies
\begin{align*}
g(t-r) - g(s-r) = \mathds{1}_{(s,t]}(r) I_n  + \int_s^t g \ast \eta (u-r) \, du
\end{align*}
for almost all $r\in \mathbb{R}$ and each fixed $s<t$. 
\item\label{gMoments} If $\eta$ has moment of order $p \in \mathbb{N}$, then $g \in L^q$ for all $q \in [1/p,\infty]$, and
\begin{align}\label{gStieltjes}
g(t) = \mathds{1}_{[0,\infty)}(t)I_n + \int_{-\infty}^t g \ast \eta (u)\, du
\end{align}
for almost all $t \in \mathbb{R}$. In particular,
\begin{align}\label{minusIdentity}
\int_\mathbb{R}g \ast \eta (u)\, du = - I_n.
\end{align}
\item\label{gExponential} If $\int_{[0,\infty)}e^{\delta u }\, \vert \eta_{jk}\vert (du) < \infty$ for all $j,k=1,\dots, n$ and some $\delta >0$, then there exists $\varepsilon >0$ such that
\begin{align*}
\sup_{t \in \mathbb{R}}\max_{j,k =1,\dots,n} \vert g_{jk}(t)\vert e^{\varepsilon \vert t \vert} \leq C 
\end{align*}
for a suitable constant $C>0$.

\item\label{gCausality} If $\det (h(z)) \neq 0$ for all $z \in \mathbb{C}$ with $\text{Re}(z)\leq 0$ then $g$ is vanishing on $(-\infty,0)$ almost everywhere. 
\end{enumerate}
\end{proposition}

\begin{proof}
In order to show the existence of $g$ it suffices to argue that
\begin{align}\label{L2requirement}
y \mapsto \big(h(iy)^{-1} \big)_{jk} \text{ is in } L^2\text{ for } j,k=1,\dots, n,
\end{align}
since the Fourier transform $\mathcal{F}$ maps $L^2$ onto $L^2$. (Here $(h(iy)^{-1})_{jk}$ refers to the $(j,k)$-th entry in the matrix $h(iy)^{-1}$.) Indeed, in this case we just set $g_{jk} = \mathcal{F}^{-1}[(h(i\cdot)^{-1})_{jk}]$. 

Let $\widehat{h(iy)}$ denote the $n\times n$ matrix which has the same rows as $h(iy)$, but where the $j$-th column is replaced by the $k$-th canonical basis vector (that is, the vector with all entries equal to zero except of the $k$-th entry which equals one). Then it follows by Cramer's rule that
\begin{align*}
\big(h(iy)^{-1} \big)_{jk} = \frac{\det (\widehat{h(iy)})}{\det (h(iy))}.
\end{align*}
Recalling that $h(iy) = -iyI - \mathcal{F}[\eta](y)$ and that $\mathcal{F}[\eta](y)$ is bounded in $y$ we get by the Leibniz formula that $\vert \det (h(iy))\vert \sim \vert y\vert^n$ and $\vert \det ( \widehat{h(iy)})\vert = O( \vert y\vert^{n-1})$ as $\vert y\vert\to \infty$. This shows in particular that
\begin{align}\label{hInvAsymp}
\big\vert \big(h(iy)^{-1} \big)_{jk}\big\vert =O\big( \vert y \vert ^{-1}\big)
\end{align}
as $\vert y \vert \to \infty$. Since $j$ and $k$ were arbitrarily chosen we get by continuity of (all the entries of) $y \mapsto h(iy)^{-1}$ that \eqref{L2requirement} holds, which ensures the existence part. The fact that $\overline{\mathcal{F}[g](-y)}  = \mathcal{F}[g](y)$, $y \in \mathbb{R}$, implies that $g$ takes values in $\mathbb{R}^{n \times n}$.

To show \eqref{functionRelation}, we fix $s<t$ and apply the Fourier transform to obtain
\begin{align*}
\MoveEqLeft\mathcal{F}\biggr[g(t-\cdot) -g(s-\cdot) - \int_s^t g\ast \eta (u-\cdot)\, du\biggr] (y) \\
&= (e^{ity}-e^{isy}) \mathcal{F}[g] (-y) - \mathcal{F}[\mathds{1}_{(s,t]}](y)\mathcal{F}[g](-y)\mathcal{F}[\eta](-y) \\
&= \mathcal{F}[\mathds{1}_{(s,t]}](y) h(-iy)^{-1}(iyI-\mathcal{F}[\eta](-y)) \\
&= \mathcal{F}[\mathds{1}_{(s,t]}](y)I_n,
\end{align*}
which verifies the result. 

We will now show \eqref{gMoments} and for this we suppose that $\eta$ has a moment of order $p \in \mathbb{N}$. Then it follows that $\tilde{h}:y \mapsto h(iy)$ is (entry-wise) $p$ times differentiable with the $m$-th derivative given by 
\begin{align*}
 -\biggr(i\delta_0(\{m-1\}) + i^m\int_{[0,\infty)}e^{iuy}u^m\eta_{jk}(du)\biggr), \quad m=1,\dots, p,
\end{align*}
and in particular all the the entries of $(D^m\tilde{h})(y)$ are bounded in $y$. Observe that, clearly, if a function $A:\mathbb{R}\to \mathbb{C}^{n\times n}$ takes the form
\begin{align}\label{boundedFunc}
A(t) = B(t)C(t) D(t), \quad t \in \mathbb{R},
\end{align}
where all the entries of $B,D:\mathbb{R}\to \mathbb{C}^{n \times n}$ decay at least as $\vert y \vert^{-1}$ as $\vert y \vert \to \infty$ and all the entries of $C:\mathbb{R}\to \mathbb{C}^{n \times n}$ are bounded, then all the entries of $A$ decay at least as $\vert y \vert^{-1}$ as $\vert y \vert \to \infty$. Using the product rule for differentiation and the fact that
\begin{align*}
\big(D\tilde{h}^{-1}\big)(y) = -\tilde{h}(y)^{-1}(D\tilde{h})(y)\tilde{h}(y)^{-1}, \quad y \in \mathbb{R},
\end{align*}
it follows recursively that $D^m\tilde{h}^{-1}$ is a sum of functions of the form \eqref{boundedFunc}, thus all its entries decay at least as $\vert y \vert^{-1}$ as $\vert y \vert \to \infty$, for $m=1,\dots, p$. Since the entries of $D^m\tilde{h}^{-1}$ are continuous as well, they belong to $L^2$, and we can use the inverse Fourier transform $\mathcal{F}^{-1}$ to conclude that
\begin{align*}
\mathcal{F}^{-1}[D^p\tilde{h}] (t) = (it)^p \mathcal{F}^{-1}[\tilde{h}](t) = (it)^p g(t), \quad t \in \mathbb{R},
\end{align*}
is an $L^2$ function. This implies in turn that $t \mapsto g_{jk}(t)(1+ \vert t \vert)^p \in L^2$ and, thus,
\begin{align*}
\int_\mathbb{R}\vert g_{jk}(t)\vert^q\, dt\leq \biggr(\int_\mathbb{R}\big(g_{jk}(t)(1+ \vert t \vert)^p \big)^2\, dt \biggr)^{\tfrac{q}{2}} \biggr(\int_\mathbb{R} (1+ \vert t \vert)^{-\tfrac{2pq}{2-q}}\, dt \biggr)^{1-\tfrac{q}{2}}< \infty
\end{align*}
for any $q \in [1/p,2)$ and $j,k=1, \dots , n$. By using the particular observation that $g \in L^1$ and \eqref{functionRelation} we obtain that
\begin{align}\label{gFunctionalEq}
g(t) = \mathds{1}_{[0,\infty)}(t)I + \int_{-\infty}^t g\ast \eta (u)\, du
\end{align}
for (almost) all $t\in \mathbb{R}$. This shows that
\begin{align*}
\vert g_{jk}(t)\vert \leq 1 + \int_\mathbb{R} \vert (g\ast \eta (u))_{jk} \vert\, du
\leq 1 + \sum_{l=1}^n \int_\mathbb{R} \vert g_{jl}(u)\vert\, du\, \vert \eta_{lk}\vert ([0,\infty))
\end{align*}
for all $t\in \mathbb{R}$ and for every $j,k=1,\dots, n$ which implies $g \in L^\infty$ and, thus, $g\in L^q$ for all $q \in [1/p,\infty]$. Since $g(t) \to 0$ entrywise as $t\to \infty$, we get by \eqref{gFunctionalEq} that
\begin{align*}
\int_\mathbb{R} g \ast \eta (u)\, du = - I_n,
\end{align*}
which concludes the proof of \eqref{gMoments}.

Now suppose that $\int_{[0,\infty)}e^{\delta u}\, \vert \eta_{jk}\vert (du) < \infty$ for all $j,k = 1,\dots, n$ and some $\delta >0$. In this case, $\mathcal{S}_\delta :=\{z \in \mathbb{C}\, :\, \text{Re}(z) \in [-\delta,\delta]\}\subseteq D(\eta)$ and 
\begin{align*}
z \mapsto \det (h(z)) = \det \biggr(-zI-\int_{[0,\infty)} e^{zu}\, \eta (du)\biggr)
\end{align*}
is strictly separated from $0$ when $\vert z \vert$, $z\in \mathcal{S}_\delta$, is sufficiently large. Indeed, the dominating term in $\det (h(z))$ is $(-1)^nz^n$ when $\vert z \vert$ is large, since
\begin{align*}
\biggr\vert\biggr(\int_{[0,\infty)} e^{zu}\, \eta (du) \biggr)_{jk}\biggr\vert\leq \max_{l,m=1,\dots, n} \int_{[0,\infty)} e^{\delta u }\, \vert \eta_{lm}\vert (du)
\end{align*}
for $j,k=1,\dots, n$. Using this together with the continuity of $z \mapsto \det (h(z))$ implies that there exists $\tilde{\delta} \in (0,\delta]$ so that $z \mapsto \det (h(z))$ is strictly separated from $0$ on $\mathcal{S}_{\tilde{\delta}}:= \{z \in \mathbb{C}\, :\, \text{Re}(z)\in [-\tilde{\delta},\tilde{\delta}]\}$. In particular, $z \mapsto (h(z)^{-1})_{jk}$ is bounded on any compact set of $\mathcal{S}_{\tilde{\delta}}$, and by using Cramer's rule and the Leibniz formula as in \eqref{hInvAsymp} we get that $\vert (h(z)^{-1})_{jk}\vert =O( \vert z \vert^{-1})$ as $\vert z \vert \to \infty$ provided that $z \in \mathcal{S}_{\tilde{\delta}}$. Consequently,
\begin{align*}
\sup_{x \in [-\tilde{\delta},\tilde{\delta}]} \int_\mathbb{R}\big\vert \big(h(x+iy)^{-1}\big)_{jk} \big\vert^2\, dy <\infty,
\end{align*}
and this implies by \cite[Lemma~5.1]{contARMAframework} that $t\mapsto g_{jk}(t)e^{\varepsilon t}\in L^1$ for all $\varepsilon \in (-\tilde{\delta},\tilde{\delta})$. Fix any $\varepsilon \in (0,\tilde{\delta})$ and $j,k \in \{1, \dots, n\}$, and observe from \eqref{gFunctionalEq} that $g_{jk}$ is absolutely continuous on both $[0,\infty)$ and $(-\infty,0)$ with density $(g \ast \eta)_{jk}$. Consequently, for fixed $t >0$, integration by parts yields
\begin{align}\label{IBPpositive}
\vert g_{jk}(t)\vert e^{\varepsilon t} \leq \vert g_{jk}(0)\vert + \int_\mathbb{R} \vert (g\ast \eta (u))_{jk} \vert e^{\varepsilon u}\, du + \varepsilon \int_\mathbb{R} \vert g_{jk}(u)\vert e^{\varepsilon u}\, du.
\end{align}
Since
\begin{align*}
\int_\mathbb{R} \vert (g\ast \eta (u))_{jk} \vert e^{\varepsilon u}\, du \leq \sum_{l=1}^n \int_\mathbb{R}\vert g_{jl}(u)\vert  e^{\varepsilon u}\, du\, \int_{[0,\infty)} e^{\varepsilon u}\, \vert \eta_{lk}\vert (du)
\end{align*}
it follows from \eqref{IBPpositive} that
\begin{align*}
\max_{j,k = 1, \dots , n} \vert g_{jk} (t)\vert \leq C e^{-\varepsilon t}
\end{align*}
for all $t >0$ with
\begin{align*}
C&:= 1 \\
&+ \max_{j,k  = 1 ,\dots , n} \biggr(\sum_{l=1}^n\int_\mathbb{R} \vert g_{jl}(u)\vert e^{\varepsilon \vert u\vert }\, du\, \int_{[0,\infty)} e^{\varepsilon u }\, \vert \eta_{lk} \vert (du) + \varepsilon \int_\mathbb{R}\vert g_{jk}(u)\vert e^{\varepsilon \vert u \vert}\, du \biggr).
\end{align*}
By considering $-\varepsilon$ rather than $\varepsilon$ in the above calculations one reaches the conclusion that
\begin{align*}
\max_{j,k = 1, \dots , n} \vert g_{jk} (t)\vert \leq C e^{\varepsilon t}, \quad t<0,
\end{align*}
and this verifies \eqref{gExponential}. 

Finally, suppose that $\det (h(z))\neq 0$ for all $z \in \mathbb{C}$ with $\text{Re}(z) \leq 0$. Then it holds that $h$, and thus $z\mapsto h(z)^{-1}$, is continuous on $\{z\in \mathbb{C}\, :\, \text{Re}(z)\leq 0\}$ and analytic on $\{z\in \mathbb{C}\, :\, \text{Re}(z)< 0\}$. Moreover, arguments similar to those in \eqref{hInvAsymp} show that $\vert (h(z)^{-1})_{jk}\vert = O(\vert z\vert^{-1})$ as $\vert z \vert \to \infty$, and thus we may deduce that
\begin{align*}
\sup_{x<0}\int_\mathbb{R} \vert (h(x+iy)^{-1})_{jk}\vert\, dy <\infty. 
\end{align*}
From the theory on Hardy spaces, see \cite[Lemma~5.1]{contARMAframework}, \cite[Section~2.3]{Doetsch} or \cite{Dym:gaussian}, this implies that $g$ is vanishing on $(-\infty,0)$ almost everywhere, which verifies  \eqref{gCausality} and ends the proof.
\end{proof}

From Proposition~\ref{gExistence} it becomes evident that we may and, thus, do choose the kernel $g$ to satisfy \eqref{gStieltjes} pointwise, so that the function induces a finite Lebesgue-Stieltjes measure $g(du)$. We summarize a few properties of this measure in  the corollary below.

\begin{corollary}\label{gMeasure} Let $h$ be the function introduced in \eqref{DefOfh} and suppose that $\det (h(iy))\neq 0$ for all $y \in \mathbb{R}$. Suppose further that $\eta$ has first moment. Then the kernel $g:\mathbb{R}\to \mathbb{R}^{n\times n}$ characterized in \eqref{gKernelChar} induces an $n \times n$  finite Lebesgue-Stieltjes measure, which is given by
\begin{align}\label{specMeasure}
g(du) = I_n\delta_0 (du) + g\ast \eta (u)\, du.
\end{align}
A function $f=(f_{jk}):\mathbb{R}\to \mathbb{C}^{m \times n}$ is in $L^1(g(du))$ if 
\begin{align*}
\int_\mathbb{R} \vert f_{jl}(u) (g\ast \eta)_{lk}(u)\vert \, du < \infty, \quad l = 1, \dots, n,
\end{align*}
for $j=1,\dots, m$ and $k = 1, \dots , n$. Moreover, the measure $g(du)$ has $(p-1)$-th moment whenever $\eta$ has $p$-th moment for any $p \in \mathbb{N}$.
\end{corollary}

\begin{proof}
The fact that $g$ induces a Lebesgue-Stieltjes measure of the form \eqref{specMeasure} is an immediate consequence of \eqref{gStieltjes}. For a measurable function $f=(f_{jk}):\mathbb{R}\to \mathbb{C}^{m\times n}$ to be integrable with respect to $g(du) = (g_{jk}(du))$ we require that $f_{jl} \in L^1(\vert g_{lk}(du)\vert)$, $l=1,\dots, n$, for each choice of $j=1,\dots, m$ and $k=1,\dots, n$. Since the variation measure $\vert g_{lk}\vert (du)$ of $g_{lk}(du)$ is given by
\begin{align*}
\vert g_{lk}\vert(du) = \delta_0(\{l-k\})\delta_0(du) + \vert (g \ast \eta (u))_{lk}\vert\, du,
\end{align*}
we see that this condition is equivalent to the statement in the result. Finally, suppose that $\eta$ has $p$-th moment for some $p \in \mathbb{N}$. Then, for any $j,k \in \{1,\dots, n\}$, we get that
\begin{align*}
\int_\mathbb{R} \vert  u \vert^{p-1}\, \vert g_{jk}\vert (du) \leq \sum_{l=1}^n\bigg( &\vert \eta_{lk}\vert ([0,\infty)) \int_\mathbb{R} \vert u^{p-1} g_{jl}(u)\vert\, du \\
&+ \int_{[0,\infty)} \vert v \vert^{p-1}\, \vert \eta_{lk} \vert (dv)\, \int_\mathbb{R}  \vert g_{jl}(u)\vert\, du\biggr).
\end{align*}
From the assumptions on $\eta$ and Proposition~\ref{gExistence}(\ref{gMoments}) we get immediately that $\vert \eta_{lk}\vert ([0,\infty))$, $\int_{[0,\infty)}\vert v \vert^{p-1}\, \vert \eta_{lk} \vert (dv)$ and $\int_\mathbb{R}  \vert g_{jl}(u)\vert\, du$ are finite for all $l=1,\dots, n$. Moreover, for any such $l$ we compute that
\begin{align*}
\MoveEqLeft\int_\mathbb{R}\vert u^{p-1} g_{jl}(u)\vert\, du \\
&\leq \int_{\{\vert u \vert \leq 1\}}\vert u^{p-1}g_{jl}(u)\vert\, du +\biggr(\int_{\{\vert u \vert >1\}} u^{-2}\, du \biggr)^{\tfrac{1}{2}}\biggr( \int_{\{\vert u \vert >1\}}( u^p g_{jl}(u))^2 \, du\biggr)^{\tfrac{1}{2}}
\end{align*}
which is finite since $u \mapsto u^p g_{jl}(u)\in L^2$, according to the proof of Proposition~\ref{gExistence}(\ref{gMoments}), and hence we have shown the last part of the result. 
\end{proof}

We now give a result that both will be used to prove the uniqueness part of Theorem~\ref{existence} and Theorem~\ref{Prediction}. 

\begin{lemma}\label{variationOfConstants} Suppose that $\det (h(iy)) \neq 0$ for all $y \in \mathbb{R}$ and that $\eta$ is a finite measure with second moment, and let $g$ be given by \eqref{gKernelChar1}. Furthermore, let $(X_t)_{t\in \mathbb{R}}$ be a measurable process, which is bounded in $L^1(\mathbb{P})$ and satisfies \eqref{MSDDE} almost surely for all $s<t$. Then, for each $s\in \mathbb{R}$ and almost surely,
\begin{align}\label{varOfConstants}
\begin{aligned}
X_t =&\,  g(t-s) X_s +  \int_s^\infty g(t-u) \, \eta \ast \big\{\mathds{1}_{(-\infty,s]}X\big\} (u)\, du\\
 &+ g \ast \big\{\mathds{1}_{(s,\infty)}(Z-Z_s) \big\}(t) 
\end{aligned}
\end{align}
for Lebesgue almost all $t>s$, using the notation 
\begin{align*}
(\eta \ast \{\mathds{1}_{A}X\})_j (t) &:= \sum_{k=1}^n \int_{[0,\infty)} \mathds{1}_A (t-u)X^k_{t-u}\, \eta_{jk}(du) \quad \text{and} \\
(g\ast \{\mathds{1}_{(s,\infty)}(Z-Z_s) \})_j(t) &:=\sum_{k=1}^n \int_\mathbb{R}\mathds{1}_{(s,\infty)}(t-u)\big(Z^k_{t-u}-Z^k_s \big)\, g_{jk}(du)
\end{align*}
for $j=1,\dots, n$ and $t\in \mathbb{R}$. 
\end{lemma}

\begin{proof}
By arguments similar to those in the proof of Proposition~\ref{gExistence}(\ref{gExponential}) we get that the assumption $\det (h(iy)) \neq 0$ for all $y \in \mathbb{R}$ implies that we can choose $\delta\in (0,\varepsilon)$, such that $\det (h(z))\neq 0$ for all $z \in \mathbb{C}$ with $-\delta<\text{Re}(z) \leq 0$ and
\begin{align*}
\sup_{x \in (-\delta ,0)} \int_\mathbb{R}\big\vert \big(h(x+iy)^{-1}\big)_{jk}\big\vert^2\, dy<\infty.
\end{align*}
for all $j,k=1,\dots, n$. Thus, \cite[Lemma~5.1]{contARMAframework} ensures that $\mathcal{L}[g](z)  = h(z)^{-1}$ when $\text{Re}(z) \in (-\delta,0)$. From this point we will fix such $z$ and let $s\in \mathbb{R}$ be given. Since $(X_t)_{t\in\mathbb{R}}$ satisfies \eqref{MultiSDDEcompact}, 
\begin{align*}
\mathds{1}_{(s,\infty)}(t)X_t = \mathds{1}_{(s,\infty)}(t)X_s + \int_{-\infty}^t\mathds{1}_{(s,\infty)}(u)\, \eta \ast X (u)\, du + \mathds{1}_{(s,\infty)}(t)(Z_t-Z_s)
\end{align*}
for Lebesgue almost all $t\in \mathbb{R}$ outside a $\mathbb{P}$-null set (which is a consequence of Tonelli's theorem). In particular, this shows that
\begin{align*}
\MoveEqLeft -z\mathcal{L}[\mathds{1}_{(s,\infty)}X](z) \\
=& -z\biggr\{ X_s\mathcal{L}[\mathds{1}_{(s,\infty)}](z) + \mathcal{L}\biggr[\int_{-\infty}^\cdot \mathds{1}_{(s,\infty)}(u)\,  \eta \ast X (u)\, du \biggr](z)\\
&+ \mathcal{L}[\mathds{1}_{(s,\infty)} (Z-Z_s)](z) \biggr\}\\
=& \mathcal{L}[X_s\delta_0(\cdot - s)](z) + \mathcal{L} [\mathds{1}_{(s,\infty)}\, \eta \ast X](z) - z\mathcal{L}[\mathds{1}_{(s,\infty)} (Z-Z_s)](z).
\end{align*}
By noticing that 
\begin{align*}
\mathcal{L}[\mathds{1}_{(s,\infty)}\, \eta \ast X](z) &= \mathcal{L} \big[\mathds{1}_{(s,\infty)}\, \eta \ast \big\{\mathds{1}_{(-\infty,s]}X\big\}\big](z) +  \LL \big[ \eta \ast \big\{\mathds{1}_{(s,\infty)}X\big\}\big](z)\\
&= \mathcal{L} \big[\mathds{1}_{(s,\infty)}\, \eta \ast \big\{\mathds{1}_{(-\infty,s]}X\big\}\big](z) +  \LL[ \eta](z)\LL[\big\{\mathds{1}_{(s,\infty)}X\big\}\big](z)
\end{align*}
it thus follows that
\begin{align*}
\MoveEqLeft h(z) \mathcal{L}[\mathds{1}_{(s,\infty)} X](z) \\
=&  \mathcal{L}\big[X_s\delta_0(\cdot - s) + \mathds{1}_{(s,\infty)}\, \eta \ast \big\{\mathds{1}_{(-\infty,s]}X\big\} \big](z) - z \mathcal{L}[\mathds{1}_{(s,\infty)}(Z-Z_s)](z).
\end{align*}
(The reader should observe that since both $(X_t)_{t\in \mathbb{R}}$ and $(Z_t)_{t\in \mathbb{R}}$ are bounded in $L^1(\mathbb{P})$, the Laplace transforms above are all well-defined almost surely. We refer to the beginning of the proof of Theorem~\ref{existence} where details for a similar argument are given.) Now, using that $\mathcal{L}[g](z) = h(z)^{-1}$, we notice
\begin{align*}
-zh(z)^{-1}\mathcal{L}[\mathds{1}_{(s,\infty)}(Z-Z_s)](z) &= \mathcal{L}[g(du)](z) \mathcal{L}[\mathds{1}_{(s,\infty)}(Z-Z_s)](z) \\
&= \mathcal{L}\big[g \ast \big\{\mathds{1}_{(s,\infty)}(Z-Z_s) \big\}\big] (z),
\end{align*}
and thus
\begin{align*}
X_t =&\,  g(t-s) X_s + \int_s^\infty g(t-u) \, \eta \ast \big\{\mathds{1}_{(-\infty,s]}X\big\} (u)\, du + g \ast \big\{\mathds{1}_{(s,\infty)}(Z-Z_s) \big\}
\end{align*}
for Lebesgue almost all $t>s$ with probability one. 

\end{proof}
With Lemma~\ref{variationOfConstants} in hand we are now ready to prove the general result, Theorem~\ref{existence}, for existence and uniqueness of solutions to the MSDDE \eqref{MultiSDDEcompact}.


\begin{proof}[Proof of Theorem~\ref{existence}]
Fix $t\in \mathbb{R}$. The convolution in \eqref{solutionForm} is well-defined if $u \mapsto Z^T_{t-u}$ is $g^T$-integrable (by Corollary~\ref{gMeasure}) which means that $u \mapsto Z^k_{t-u}$ belongs to $L^1(\vert g_{jk}\vert(du))$ for all $j,k = 1,\dots, n$. Observe that, since $(Z^k_u)_{u \in\mathbb{R}}$ is integrable and has stationary increments, \cite[Corollary~A.3]{QOU} implies that there exists $\alpha,\beta>0$ such that $\mathbb{E}[\vert Z_u^k\vert]\leq \alpha + \beta \vert u \vert$ for all $u \in \mathbb{R}$. Consequently,
\begin{align*}
\mathbb{E}\biggr[\int_\mathbb{R}\vert Z_{t-u}^k \vert\, \mu (du) \biggr] \leq (\alpha + \beta \vert t \vert) \mu (\mathbb{R}) + \beta\int_{\mathbb{R}} \vert u \vert\, \mu (du) <\infty
\end{align*}
for any (non-negative) measure $\mu$ which has first moment. This shows that $u \mapsto Z^k_{t-u}$ will be integrable with respect to such measure almost surely, in particular with respect to $\vert g_{jk}\vert (du)$, $j=1,\dots ,n$, according to Corollary~\ref{gMeasure} as $\eta$ has second moment.

We will now argue that $(X_t)_{t\in \mathbb{R}}$ defined by \eqref{solutionForm} does indeed satisfy \eqref{MultiSDDEcompact}, and thus we fix $s<t$. Due to the fact that
\begin{align*}
\int_s^t X^T\ast \eta^T (u)\, du = \int_s^t Z^T\ast \eta^T (u)\, du + \int_s^t \biggr(\int_\mathbb{R} g\ast \eta (r) Z_{\cdot - r}\, du\biggr)^T\ast \eta^T (u)\, du
\end{align*}
it is clear by the definition of $(X_t)_{t\in \mathbb{R}}$ that it suffices to argue that
\begin{align*}
\MoveEqLeft\int_s^t \biggr(\int_\mathbb{R} g\ast \eta (r) Z_{\cdot - r}\, du\biggr)^T\ast \eta^T (u)\, du \\
&= \int_\mathbb{R} Z^T_r[g\ast \eta (t-r) - g\ast \eta (s-r)]^T\, dr - \int_s^tZ^T\ast \eta^T(r)\, dr.
\end{align*}
We do this componentwise, so we fix $i \in \{1, \dots, n\}$ and compute that
\begin{align*}
\MoveEqLeft\biggr(\int_s^t \biggr(\int_\mathbb{R}g \ast \eta (r)Z_{\cdot - r}\, dr \biggr)^T\ast \eta^T(u)\, du \biggr)_i\\
=& \sum_{j=1}^n\sum_{k=1}^n\sum_{l=1}^n \int_s^t \biggr( \int_\mathbb{R}g_{jl}\ast \eta_{lk}(v)Z^k_{\cdot - r}\, dr \biggr)\ast \eta_{ij}(u)\, du\\
=& \sum_{j=1}^n\sum_{k=1}^n\sum_{l=1}^n\int_\mathbb{R}Z^k_r\int_{[0,\infty)}\int_s^t \int_{[0,\infty)} g_{jl}(u-v - r - w)\, \eta_{ij}(dv)\, du\, \eta_{lk}(dw)\, dr\\
=& \sum_{k=1}^n\sum_{l=1}^n \int_\mathbb{R} Z_r^k \int_{[0,\infty)}\int_s^t (g\ast \eta)_{il}(u-r-w)\, du\, \eta_{lk}(dw)\, dr\\
=& \sum_{k=1}^n\sum_{l=1}^n \biggr(\int_\mathbb{R} Z^k_r \int_{[0,\infty)}[g_{il}(t-r-w) - g_{il}(s-r-w)]\, \eta_{lk}(dw)\, dr\\
& -\int_\mathbb{R} Z^k_r \int_{[0,\infty)}\delta_0(\{i-l\})\mathds{1}_{(s,t]}(r+w)\, \eta_{lk}(dw)\, dr \biggr)\\
=& \sum_{k=1}^n \biggr(\int_\mathbb{R}Z^k_r [(g \ast \eta)_{ik}(t-r)-(g \ast \eta)_{ik}(s-r)]\, dr- \int_s^t Z^k\ast \eta_{ik} (r)\, dr \biggr)\\
=& \biggr(\int_\mathbb{R} Z_r^T [g\ast \eta (t-r)-g \ast \eta (s-r)]^T\, dr -\int_s^t Z^T\ast \eta^T (r)\, dr \biggr)_i
\end{align*}
where we have used \eqref{functionRelation} in Proposition~\ref{gExistence} and the fact that $g$ and $\eta$ commute in a convolution sense, $g\ast \eta = (g^T \ast \eta^T)^T$ (compare the associated Fourier transforms).

Next, we need to argue that $(X_t)_{t\in\mathbb{R}}$ is stationary. Here we will use \eqref{minusIdentity} to write the solution as
\begin{align*}
X_t = \int_\mathbb{R}g\ast \eta (u)\, [Z_{t-u}-Z_t]\, du
\end{align*}
for each $t\in \mathbb{R}$. Fix $m \in \mathbb{R}$. Let $-m=t^k_0<t^k_1<\dots<t^k_k = m$ be a partition of $[-m,m]$ with $\max_{j=1,\dots, k}( t^k_j-t^k_{j-1})\to 0$, $k \to \infty$, and define the Riemann sum
\begin{align*}
X^{m,k}_t = \sum_{j=1}^k g\ast \eta (t^k_{j-1})\, [Z_{t-t^k_{j-1}}-Z_t]\,  (t^k_j-t^k_{j-1}).
\end{align*}
Observe that $(X^{m,k}_t)_{t\in \mathbb{R}}$ is stationary. Moreover, the $i$-th component of $X_t^{m,k}$ converges to the $i$-th component of
\begin{align*}
X^m_t= \int_{-m}^m g \ast \eta (u)\, [Z_{t-u}-Z_t]\, du
\end{align*}
in $L^1(\mathbb{P})$ as $k\to \infty$. To see this, we start by noting that
\begin{align*}
\mathbb{E}\big[\big\vert \big(X^m_t\big)_i - \big(X^{m,k}_t\big)_i \big\vert\big] \leq &
\sum_{j=1}^n \int_\mathbb{R}\sum_{l=1}^k\mathds{1}_{(t^k_{l-1},t^k_l]}(u)\mathbb{E}\Big[\big\vert (g\ast \eta)_{ij}(u) \big[Z^j_{t-u}-Z^j_t \big]\\&-(g\ast \eta)_{ij}\big(t_{l-1}^k\big)\big[Z^j_{t-t_{l-1}^k}-Z_t^j \big] \big\vert \Big]\, du.
\end{align*}
Then, for each $j \in \{1,\dots, n\}$,
\begin{align*}
 \max_{l=1,\dots, k}\mathds{1}_{(t^k_{l-1},t_l^k]}(u)\MoveEqLeft\mathbb{E}\Big[\big\vert (g\ast \eta)_{ij}(u) \big[Z^j_{t-u}-Z^j_t \big]-(g\ast \eta)_{ij}\big(t_{l-1}^k\big)\big[Z^j_{t-t_{l-1}^k}-Z_t^j \big] \big\vert \Big]\\
 \leq & \max_{l=1,\dots, k}\mathds{1}_{(t_{l-1}^k,t_l^k]}(u) \Big(\vert(g\ast \eta)_{ij}(u)\vert\, \mathbb{E}\big[\big\vert Z^j_{t-u} - Z^j_{t-t_{l-1}^k}\big\vert\big]\\
 &+ \mathbb{E}\big[\big\vert Z^j_{t-t_{l-1}^k} - Z^j_t \big\vert \big]\, \big\vert(g\ast \eta)_{ij}(u) - (g\ast \eta)_{ij}\big(t_{l-1}^k\big)\big\vert \Big) \to 0
\end{align*}
as $k \to \infty$ for almost all $u \in\mathbb{R}$ using that $(Z^j_t)_{t\in \mathbb{R}}$ is continuous in $L^1(\mathbb{P})$ (cf. \cite[Corollary~A.3]{QOU}) and that $(g\ast \eta)_{ij}$ is càdlàg. Consequently, Lebesgue's theorem on dominated convergence implies that $X^{m,k}_t \to X^m_t$ entrywise in $L^1(\mathbb{P})$ as $k\to \infty$, thus $(X^m_t)_{t\in \mathbb{R}}$ inherits the stationarity property from $(X^{m,k}_t)_{t\in \mathbb{R}}$. Finally, since $X^m_t \to X_t$ (entrywise) almost surely as $m \to\infty$, we obtain that $(X_t)_{t\in \mathbb{R}}$ is stationary as well.

To show the uniqueness part, we let $(U_t)_{t \in \mathbb{R}}$ and $(V_t)_{t \in \mathbb{R}}$ be two stationary, integrable and measurable solutions to \eqref{MultiSDDEcompact}. Then $X_t := U_t - V_t$, $t\in \mathbb{R}$, is bounded in $L^1(\mathbb{P})$ and satisfies an MSDDE without noise. Consequently, Lemma~\ref{variationOfConstants} implies that 
\begin{align*}
X_t = g(t-s) X_s + \int_s^\infty g(t-u) \eta \ast \{ \mathds{1}_{(-\infty,s]}X\}(u)\, du
\end{align*}
holds for each $s \in \RR$ and Lebesgue almost all $t >s$. For a given $j$ we thus find that 
\begin{align*}
\mathbb{E}\big[\big\vert X^j_t\big\vert\big] &\leq C\sum_{k=1}^n\biggr(\vert g_{jk}(t-s) \vert 
+ \sum_{l=1}^n \int_s^\infty \vert g_{jk}(t-u)\vert\, \vert \eta_{kl}\vert ([u-s,\infty))\, du
\biggr)
\end{align*}
where $C:= \max_k \mathbb{E}[\vert U_0^k\vert + \vert V_0^k\vert]$. It follows by Proposition~\ref{gExistence}(ii) that $g(t)$ converges as $t \to \infty$, and since $g \in L^1$ it must be towards zero. Using this fact together with Lebesgue's theorem on dominated convergence it follows that the right-hand side of the expression above converges to zero as $s$ tends to $-\infty$, from which we conclude that $U_t = V_t$ almost surely for Lebesgue almost all $t$. By continuity of both processes in $L^1(\mathbb{P})$ (cf. \cite[Corollary~A.3]{QOU}), we get the same conclusion for all $t$. 

Finally, under the assumption that $\det(h(z)) \neq 0$ for $z \in \CC$ with $\Real(z) \leq 0$ it follows from Proposition~\ref{gExistence}(\ref{gCausality}) that $g\ast \eta$ is vanishing on $(-\infty,0)$, and hence we get that the solution $(X_t)_{t\in \mathbb{R}}$ defined by \eqref{solutionForm} is causal since
\begin{align*}
X_t = Z_t + \int_0^\infty g \ast \eta (u)\, Z_{t-u}\, du = - \int_0^\infty g\ast \eta (u) [Z_t - Z_{t-u}]\, du
\end{align*}
for $t\in \mathbb{R}$ by \eqref{minusIdentity}.
\end{proof}

\begin{proof}[Proof of Theorem \ref{Prediction}]
Since $(X_t)_{t \in \RR}$ is a solution to an MSDDE, 
\begin{align*}
\sigma (X_u : u \leq s) = \sigma (Z_s -Z_u : u \leq s)
\end{align*}
and the theorem therefore follows by Lemma~\ref{variationOfConstants}.  
\end{proof}

\begin{proof}[Proof of Proposition \ref{MArep}]
We start by arguing why \eqref{FubiniRelation} is well-defined. To see that this is the case, note initially that $I_k(f_r(t-\cdot)) = Z_t^k-Z_{t-r}^k$ and thus, since $(Z^k_t)_{t\in \mathbb{R}}$ is integrable and has stationary increments, there exists $\alpha,\beta >0$ such that $\mathbb{E}[\vert I_k(f_r(t-\cdot))\vert]\leq \alpha + \beta \vert r \vert$ for all $r \in \mathbb{R}$ (see, e.g., \cite[Corollary~A.3]{QOU}). In particular
\begin{align*}
\mathbb{E}\biggr[\int_\mathbb{R}\vert I_k(f_r(t-\cdot))\vert\, \vert \mu \vert (dr) \biggr] \leq \alpha \vert \mu \vert (\mathbb{R}) + \beta \int_\mathbb{R} \vert r \vert \, \vert \mu \vert (dr) < \infty,
\end{align*}
which shows that $I_k(f_r(t-\cdot))$ is integrable with respect to $\mu$, thus the right-hand side of \eqref{FubiniRelation} is well-defined, almost surely for each $t \in \mathbb{R}$. To show that the left-hand side is well-defined, it suffices to note that $u \mapsto \int_\mathbb{R} f_r(u)\, \mu (dr)$ belongs to $L^1 \cap L^2$ by an application of Jensen's inequality and Tonelli's theorem.

To show \eqref{MArelation} we start by fixing $t \in \mathbb{R}$ and $j,k \in \{1,\dots, n\}$, and by noting that $\mu (dr) = (g\ast \eta)_{jk}(r)\, dr$ is a finite measure with having first moment according to Corollary~\ref{gMeasure}. Consequently, we can use assumptions (i)-(ii) on $I_k$ to get
\begin{align*}
\int_\mathbb{R} (g\ast \eta)_{jk}(r)\big[Z^k_{t-r}-Z^k_t \big]\, dr &= \int_\mathbb{R} I_k(\mathds{1}_{(t,t-r]})(g\ast \eta)_{jk}(r)\, dr\\
&=I_k \biggr(\int_\mathbb{R} \mathds{1}_{(t,t-r]} (g\ast \eta)_{jk}(r)\, dr \biggr)\\
&= I_k\biggr(\delta_0(\{j-k\})\mathds{1}_{[0,\infty)}(t-\cdot)+\int_{-\infty}^{t-\cdot} (g\ast \eta)_{jk}(u)\, du \biggr)\\
&= I_k (g_{jk}(t-\cdot))
\end{align*}
using \eqref{gStieltjes} and the convention that $\mathds{1}_{(a,b]} = -\mathds{1}_{(b,a]}$ when $a>b$. By combining this relation with \eqref{minusIdentity} and \eqref{solutionForm} we obtain
\begin{align*}
X^j_t = \sum_{k=1}^n \int_\mathbb{R} (g\ast \eta)_{jk}(r) [Z^k_{t-r}-Z^k_t]\, dr = \sum_{k=1}^n I_k (g_{jk}(t-\cdot)).
\end{align*}
\end{proof}

\begin{proof}[Proof of Proposition \ref{RiemannReg}]
Let $\alpha \in (1,2]$ and $\beta \in (0,1-1/\alpha)$, and consider a function $f:\mathbb{R}\to \mathbb{R}$ in $L^1\cap L^{\alpha}$. We start by noticing that 
\begin{align*}
\int_t^\infty \vert f(u) \vert (u-t)^{\beta -1}\, du =  \int_0^1 \vert f(t+u) \vert u^{\beta -1}\, du + \int_1^\infty \vert f(t+u) \vert u^{\beta -1}\, du.
\end{align*}
For the left term we find that  
\begin{align*}
\MoveEqLeft \int_\RR \biggr( \int_0^1 \vert f(t+u)\vert u^{\beta -1}\, du \biggr) ^\alpha\, dt \\
&\leq \biggr( \int_0^1 u^{\beta -1}\, du\biggr)^{\alpha-1} \int_\RR  \int_0^1 \vert f(t+u)\vert ^\alpha u^{\beta -1}\, du\, dt\\
& = \biggr(\int_0^1 u^{\beta -1}\, du\biggr)^\alpha \int_\RR \vert f(t) \vert^ \alpha\, dt < \infty.
\end{align*}
For the right term we find 
\begin{align*}
\MoveEqLeft \int_\RR \biggr(\int_1^\infty \vert f(t+u) \vert u^{\beta -1}\, du \biggr)^\alpha\, dt \\
&\leq \biggr( \int_\RR f(u) du \biggr)^{\alpha-1} \int_\RR \int_1^\infty \vert f(t+u) \vert u^{\alpha(\beta -1)}\, du\,  dt \\
& = \biggr( \int_\RR f(u) du \biggr)^{\alpha}  \int_1^\infty  u^{\alpha(\beta -1)}\, du <\infty.
\end{align*}
We conclude that $\big(I^{\beta}_-f\big)(u) \in L^{\alpha}$.
\end{proof}

\begin{proof}[Proof of Theorem~\ref{simpleFourierKernel}]
The identity \eqref{higherOrderDetCond} is just a matter of applying standard computation rules for determinants. For instance, one may prove the result when $z\neq 0$ by induction using the block representation
\begin{align}\label{blockH}
-h(z) = \begin{bmatrix} A & B \\ C & D \end{bmatrix}
\end{align}
with $A=I_nz$, $B= (e_1 \otimes I_n)^T \in \mathbb{R}^{n \times (m-1)n}$, $C= e_{m-1}\otimes \mathcal{L}[\varpi_0](z) \in \mathbb{R}^{(m-1)n \times n}$, and
\begin{align*}
D= \begin{bmatrix}
I_nz & I_n & 0 & \cdots & 0 \\
0 & I_n z & I_n & \cdots & 0 \\
\vdots &  \vdots & \ddots & \ddots & \vdots \\
0 & 0 & \cdots & I_n z & I_n \\
\mathcal{L}[\varpi_1](z) & \mathcal{L}[\varpi_2](z) & \cdots & \mathcal{L}[\varpi_{m-2}](z) & I_nz +\mathcal{L}[\varpi_{m-1}](z)
\end{bmatrix}.
\end{align*}
Here $e_1$ and $e_{m-1}$ refer to the firs and last canonical basis vector of $\mathbb{R}^{m-1}$, respectively. The case where $z=0$ follows directly from the Leibniz formula. In case $\det (h(iy))\neq 0$ for all $y \in \mathbb{R}$, we may write $h(iy)^{-1}$ as an $m \times m$ matrix, where each element $(h^{-1}(iy))_{jk}$ is an $n \times n$ matrix. We then know from Theorem~\ref{existence} that the unique solution to \eqref{hOrderSDDE} is a $(Z_t)_{t\in \mathbb{R}}$-driven moving average of the form \eqref{solutionTohOrder} with $\mathcal{F}[g_{1m}](y) = (h^{-1}(iy))_{1m}$. Similar to the computation of $\det (h(z))$, when $h(z)$ is invertible, block $(1,m)$ of $h(z)^{-1}$ can inductively be shown to coincide with
\begin{align*}
\biggr( I_n (-z)^m- \sum_{j=0}^{m-1} \mathcal{L}[\varpi_j] (z) (-z)^j\biggr)^{-1}
\end{align*}
using the representation \eqref{blockH} and standard rules for inverting block matrices. This means in particular that \eqref{hOrderFourierT} is true. 
\end{proof}

\begin{proof}[Proof of Theorem~\ref{MCARMAasMSDDE}]
We start by arguing that that there exists a function $f$ with the Fourier transform in \eqref{fFourier}. Note that, since $z \mapsto \det (Q(z))$ is just a polynomial (of order $nq$), the assumption that $\det (Q(z))\neq 0$ whenever $\text{Re}(z)\geq 0$ implies in fact that
\begin{align*}
H(z) := R(-z)-Q(-z)^{-1}P(-z) = Q(-z)^{-1}[Q(-z)R(-z) - P(-z)]
\end{align*}
is well-defined for all $z \in \mathcal{S}_\delta := \{x+iy\, :\, x \leq \delta,\, y \in \mathbb{R}\}$ and a suitably chosen $\delta>0$. According to \cite[Lemma~5.1]{contARMAframework} it suffices to argue that there exists $\varepsilon\in (0,\delta]$ such that
\begin{align}\label{keyfExist}
\sup_{x <\varepsilon} \int_\mathbb{R} \vert H(x+iy)_{jk}\vert^2\, dy< \infty
\end{align}
for all $j,k=1,\dots, n$. Let $\lVert \cdot \rVert$ denote any sub-multiplicative norm on $\mathbb{C}^{n\times n}$ and note that $\vert H(z)_{jk} \vert \leq \lVert Q(-z)^{-1} \rVert \lVert Q(-z)R(-z) - P(-z) \rVert$. Thus, since $\lVert Q(z)R(z) - P(z) \rVert \sim c_1\vert z \vert^{q-1}$ and $\lVert Q(z)^{-1} \rVert \sim c_2\vert z\vert^{-q}$ as $\vert z \vert \to \infty$ for some $c_1,c_2\geq 1$ (the former by the choice of $R$ and the latter by Cramer's rule), $\vert H(z)_{jk}\vert = O(\vert z  \vert^{-1})$. Consequently, the continuity of $H$ ensures that \eqref{keyfExist} is satisfied for a suitable $\varepsilon\in (0,\delta]$, and we have established the existence of $f$ with the desired Fourier transform. This also establishes that the $n\times n$ measures $\varpi_0,\varpi_1,\dots, \varpi_{p-q-1}$ defined as in the statement of the theorem are finite and have moments of any order. Associate to these measures the $n(p-q)\times n(p-q)$ measure $\eta$ given in \eqref{etaHigherOrder}. Then it follows from \eqref{higherOrderDetCond} that
\begin{align*}
\det (h(iy)) =  \det\biggr(I_n(-iy)^{p-q} + \sum_{j=0}^{p-q-1} R_j (-iy)^j - \mathcal{F}[f](y)\biggr) = \frac{\det (P(-iy))}{\det (Q(-iy))},
\end{align*}
and hence is non-zero for all $y \in \mathbb{R}$. In light of Proposition~\ref{simpleFourierKernel}, in particular \eqref{hOrderFourierT}, we may therefore conclude that the unique solution to \eqref{hOrderSDDE} is a $(Z_t)_{t\in \mathbb{R}}$-driven moving average, where the driving kernel has Fourier transform
\begin{align*}
\biggr(I_n (-iy)+ \sum_{j=0}^{p-q-1} R_j (-iy)^j - \mathcal{F}[f](y) \biggr)^{-1} = P(-iy)^{-1}Q(-iy)
\end{align*}
for $y \in \mathbb{R}$. In other words, the unique solution is the $(Z_t)_{t\in \mathbb{R}}$-driven MCARMA($p,q$) process associated to the polynomials $P$ and $Q$.
\end{proof}

Before giving the proof of Corollary \ref{CARMAprediction} we will need the following lemma:

\begin{lemma}\label{FirstRowOfg}
Let $C_0,\dots,C_{p-q-1}$ be given in \eqref{residuePolyn} and $C_{p-q} =I_n$. Define 
\begin{align*}
R_j(z) = \sum_{k = j}^{p-q} C_k z^{k-j} , \quad j=1,\dots,p-q-1.
\end{align*}
Then $\tilde{g}$ is $p-q-2$ times differentiable and $D^{p-q-2}\tilde{g}$ has a density with respect to the Lebesgue measure which we denote $D^{p-q-1}\tilde{g}$. Furthermore, we have that
\begin{align}\label{Firstthingtoprove}
(e_1^{p-q} \otimes I_n)^T g = (\tilde{g} R_1(D),\dots,\tilde{g} R_{p-q-1}(D),\tilde{g} )
\end{align}
where 
\begin{align}\label{Secondthingtoprove}
\begin{aligned}
\tilde{g} R_j(D)(t) &= \sum_{k = j}^{p-q} D^{k-j} \tilde{g} (t) C_k  \\
& = \mathds{1}_{[0,\infty)}(t) (e_1^p\otimes I_n)^T e^{At} \sum_{k=j}^{p-q} A^{k-j} E C_k
\end{aligned}
\end{align}
for $j=1,\dots,p-q-1$ and $g : \RR \to \RR^{n\times n}$ is characterized by $\FF[g](y)=h(iy)^{-1}$ with $h: \CC \to \CC^{n(p-q)\times n(p-q)}$ given by 
\begin{align*}
h(-z)  =  \begin{bmatrix}
I_n z & -I_n & 0 & \cdots & 0 \\
0 & I_n z & -I_n & \cdots & 0 \\
\vdots & \vdots & \ddots & \ddots & \vdots \\
0 & 0 & \cdots & I_n z &  -I_n\\
 Q^{-1}(z)P(z) -zR_1(z) & C_1 & \cdots & C_{p-q-2} & I_n z+C_{p-q-1}
\end{bmatrix}.
\end{align*}
\end{lemma}

\begin{proof}

That $\tilde{g}$ is $p-q-2$ times differentiable and $D^{p-q-2}\tilde{g}$ has a density with respect to the Lebesgue measure follows form the relation in \eqref{gStieltjes}. Furthermore, by Theorem~\ref{MCARMAasMSDDE} we know that $\FF[\tilde{g} ](y) = P^{-1}(-iy)Q(-iy)$. Consequently, \eqref{Firstthingtoprove} follows since 
\begin{align*}
\MoveEqLeft (P^{-1}(-iy)Q(-iy)R_1(-iy),\\
&\dots,P^{-1}(-iy)Q(-iy)R_{p-q-1}(-iy),P^{-1}(-iy)Q(-iy) )h(z) = (e_1^{p-q} \otimes I_n)^T.
\end{align*} 
The relation in \eqref{Secondthingtoprove} follows by the representation of $\tilde{g}$ given in \eqref{gMCARMA}. 
\end{proof}

\begin{proof}[Proof of Corollary \ref{CARMAprediction}]

The prediction formula is a consequence of Lemma~\ref{FirstRowOfg} combined with Theorem~\ref{Prediction} and Theorem~\ref{MCARMAasMSDDE}. Furthermore, to get the expression for $\tilde{g} \ast \{\hat{Z}\mathds{1}_{(s,\infty)} \}$, note that 
\begin{align*}
\tilde{g} (dv) =  \mathds{1}_{\{ p=q+1 \} } \delta_0 (dv) +(e_1^p \otimes I_n)^Te^{Av}AE  \, dv,
\end{align*}
which follows from the representation of $\tilde{g}$ in \eqref{gMCARMA}
\end{proof}
\subsection*{Acknowledgments}
This work was supported by the Danish Council for Independent Research (Grant DFF - 4002 - 00003).

\bibliographystyle{chicago}

\end{document}